\documentclass[a4paper,reqno,oneside,final]{amsart}
\usepackage{a4wide}
\usepackage[notcite,notref]{showkeys}
\usepackage{amsmath,amsfonts,amssymb,amsthm}
\usepackage[utf8]{inputenc}
\usepackage{ifthen}
\usepackage[style=american]{csquotes}
\usepackage{xcolor}
\usepackage{graphicx}
\usepackage{enumitem,tikz}
\usepackage[colorlinks,bookmarks,linkcolor=black,citecolor=black]{hyperref}
\usepackage[english,algoruled,lined,noresetcount,norelsize,linesnumbered]{algorithm2e}

\let\subset\subseteq

\newtheorem{theorem}{Theorem}
\newtheorem{lemma}[theorem]{Lemma}
\newtheorem{proposition}[theorem]{Proposition}
\newtheorem{corollary}[theorem]{Corollary}

\newtheorem{claim}{Claim}

\theoremstyle{definition}
\newtheorem{definition}[theorem]{Definition}

\theoremstyle{remark}

\newcommand{\oldqed}{}
\def\endofClaim{\hfill\scalebox{.6}{$\Box$}}

\newenvironment{claimproof}[1][Proof]{
  \renewcommand{\oldqed}{\qedsymbol}
  \renewcommand{\qedsymbol}{\endofClaim}
  \begin{proof}[#1]
}{
  \end{proof}
  \renewcommand{\qedsymbol}{\oldqed}
}

\newcommand{\Ex}{\mathbb{E}}

\newcommand{\tW}{\widetilde{W}}
\newcommand{\tcW}{\widetilde{\mathcal{W}}}

\newcommand{\eps}{\varepsilon}
\renewcommand{\epsilon}{\varepsilon}
\newcommand{\NGa}{N_{\Gamma}}
\newcommand{\Vij}{V_{i,j}}
\newcommand{\epsa}{\eps^{\ast}}
\newcommand{\epsaa}{\eps^{\ast\ast}}
\newcommand{\nua}{\nu^{\ast}}

\newcommand{\kq}{q}

\newcommand{\symd}{\triangle}
\renewcommand{\Pr}{\mathbb{P}}
\newcommand{\Ca}{C^{\ast}}

\newcommand{\scalefactor}{0.5}

\setenumerate{labelindent=\parindent,leftmargin=*}

\def\itm#1{\rm ({#1})} 
\def\itmit#1{\itm{\it #1\,}} 
\def\rom{\itmit{\roman{*}}} 
\def\abc{\itmit{\alph{*}}}

\def\itmarab#1{\mbox{\itm{{\it #1\,}\arabic{*}\hspace{.05em}}}}
\def\itmarabp#1#2{\mbox{\itm{{\it #1\,}\arabic{*}#2\hspace{.05em}}}}

\newcommand{\ao}{\alpha_{\scalebox{\scalefactor}{$\mathrm{OSRIL}$}}}
\newcommand{\at}{\alpha_{\scalebox{\scalefactor}{$\mathrm{TSRIL}$}}}
\newcommand{\eo}{\eps_{\scalebox{\scalefactor}{$\mathrm{OSRIL}$}}}
\newcommand{\et}{\eps_{\scalebox{\scalefactor}{$\mathrm{TSRIL}$}}}

\newcommand{\eBL}{\eps_{\scalebox{\scalefactor}{$\mathrm{BL}$}}}
\newcommand{\CBL}{C_{\scalebox{\scalefactor}{$\mathrm{BL}$}}}

\newcommand{\cI}{\mathcal I}
\newcommand{\cJ}{\mathcal J}
\newcommand{\cL}{\mathcal L}

\newcommand{\cV}{\mathcal V}
\newcommand{\cW}{\mathcal W}

\newcommand{\Dom}{\mathrm{Dom}}

\newcommand{\im}{\mathrm{Im}}
\newcommand{\dist}{\mathrm{dist}}

\newcommand{\dcup}{\mathbin{\text{\mbox{\makebox[0mm][c]{\hphantom{$\cup$}$\cdot$}$\cup$}}}}

\newcommand\restr[2]{{
  \left.\kern-\nulldelimiterspace 
  #1 
  \vphantom{\big|} 
  \right|_{#2} 
  }}

\title[A spanning bandwidth theorem in random graphs]{A spanning bandwidth theorem in random graphs}

\author[P. Allen]{Peter Allen}
\address{(PA) London School of Economics, Department of Mathematics, Houghton Street, London WC2A 2AE, UK}
\email{p.d.allen@lse.ac.uk}
\author[J. B\"ottcher]{Julia B\"ottcher}
\address{(JB) London School of Economics, Department of Mathematics, Houghton Street, London WC2A 2AE, UK}
\email{j.boettcher@lse.ac.uk}
\author[J. Ehrenm\"uller]{Julia Ehrenm\"uller}
\address{(JE) Technische Universität Hamburg, Institut f\"ur Mathematik, Am Schwarzenberg-Campus 3, 21073 Hamburg, Germany }
\email{julia.ehrenmueller@tuhh.de}
\author[J. Schnitzer]{Jakob Schnitzer}
\address{(JS) Universität Hamburg, Fachbereich Mathematik, Bundesstraße 55, 20146 Hamburg, Germany}
\email{jakob.schnitzer@uni-hamburg.de}
\author[A. Taraz]{Anusch Taraz}
\address{(AT) Technische Universität Hamburg, Institut f\"ur Mathematik, Am Schwarzenberg-Campus 3, 21073 Hamburg, Germany }
\email{taraz@tuhh.de}

\date{\today}
\begin{document}
\begin{abstract}
 The bandwidth theorem [Mathematische Annalen, 343(1):175--205, 2009] states
that any $n$-vertex graph~$G$ with minimum degree $\big(\tfrac{k-1}{k}+o(1)\big)n$ contains all $n$-vertex $k$-colourable graphs~$H$ with bounded maximum degree and bandwidth $o(n)$. In [arXiv:1612.00661] a random graph analogue of this statement is proved: for $p\gg \big(\tfrac{\log n}{n}\big)^{1/\Delta}$ a.a.s.\ each spanning subgraph~$G$ of
$G(n,p)$ with minimum degree $\big(\tfrac{k-1}{k}+o(1)\big)pn$ contains all
$n$-vertex $k$-colourable graphs~$H$ with maximum degree $\Delta$,
bandwidth $o(n)$, and at least $C p^{-2}$ vertices not contained in any
triangle. This restriction on vertices in triangles is necessary, but limiting.

In this paper we consider how it can be avoided. A special case of our main
result is that,
under the same conditions, if additionally all vertex neighbourhoods in $G$
contain many copies of $K_\Delta$ then we can drop the restriction on $H$
that $Cp^{-2}$ vertices should not be in triangles.
\end{abstract}
\maketitle




\section{Introduction}
One major topic of research in extremal graph theory is to determine
minimum degree conditions on a graph~$G$ which force it to contain copies
of a spanning subgraph $H$. The primal example of such a theorem is Dirac's
theorem~\cite{dirac1952}, which states that if $\delta(G)\ge\tfrac12 v(G)$
then $G$ is Hamiltonian. Optimal results of this type were established for a wide range
of other spanning subgraphs~$H$ with bounded maximum degree such as powers of
Hamilton cycles, trees, or $F$-factors for any fixed graph~$F$ (see
e.g.~\cite{kuhnsurvey} for a survey). One characteristic all these graphs~$H$ have
in common is that they have sublinear
bandwidth. The \emph{bandwidth} of a labelling of the vertex set of~$H$ by integers $1, \ldots, n$
is the minimum~$b$ such that  $|i-j| \leq b$ for every edge $ij$
of~$H$. The bandwidth of~$H$ is the minimum bandwidth among all its labellings.
The relevance of this parameter was highlighted in~\cite{bottcher2009proof}, where the following asymptotically optimal general result was proved.

\begin{theorem}[Bandwidth Theorem~\cite{bottcher2009proof}]
\label{thm:bandwidth}
For every $\gamma >0$, $\Delta \geq 2$, and $k\geq 1$, there exist
$\beta>0$ and $n_0 \geq 1$ such that for every $n\geq n_0$ the following
holds. If $G$ is a graph on $n$ vertices with minimum degree $\delta(G)
\geq \left(\frac{k-1}{k}+\gamma\right)n$ and if $H$ is a $k$-colourable
graph on $n$ vertices with maximum degree $\Delta(H) \leq \Delta$ and
bandwidth at most $\beta n$, then $G$ contains a copy of $H$. \qed
\end{theorem}

More recently, the transference of extremal results from dense graphs to sparse graphs became a research focus.
Again, a prime example, due to Lee and Sudakov~\cite{lee2012}, is that if
$\Gamma=G(n,p)$ is a typical binomial random graph with $p\ge C\tfrac{\log
  n}{n}$, for some large $C$, then any $G\subset\Gamma$ with minimum degree
$\big(\tfrac12+o(1)\big)pn$ is Hamiltonian. This is a transference of
Dirac's theorem to sparse random graphs. 
Further such results exist, all focused on finding small-bandwidth
subgraphs (for a comprehensive list see, e.g., the recent survey~\cite{BCCsurv}). 
One can also ask similar questions in other sparse graphs than random
graphs---for example for sufficiently pseudorandom graphs---but we will not
focus on this question here.

As for the classical extremal statements, it is desirable to have a result
covering a very general class of spanning subgraphs. This is achieved
in~\cite{ABET}, where the following transference of the
Bandwidth Theorem to sparse random graphs is proved.

\begin{theorem}[{Sparse Bandwidth Theorem~\cite[Theorem~6]{ABET}}]
\label{thm:abet}
For each $\gamma >0$, $\Delta \geq 2$, and $k \geq 1$, there exist
constants $\beta^\ast >0$ and $\Ca >0$ such that the following holds
asymptotically almost surely for $\Gamma = G(n,p)$ if
$p \geq \Ca\big(\frac{\log n}{n}\big)^{1/\Delta}$. Let $G$ be a spanning
subgraph of $\Gamma$ with
$\delta(G) \geq\left(\frac{k-1}{k}+ \gamma\right)pn$, and let $H$ be a
$k$-colourable graph on $n$ vertices with $\Delta(H) \leq \Delta$,
bandwidth at most $\beta^\ast n$, and with at least $\Ca p^{-2}$ vertices
which are not contained in any triangles of $H$. Then $G$ contains a copy
of $H$. \qed
\end{theorem}

Note however that this result is not quite what one would expect as a
transference of the Bandwidth Theorem. There is an additional restriction
that some vertices of $H$ may not be in triangles. This restriction is
necessary, since in a sparse random graph an adversary who creates $G$ from
$\Gamma$ can typically remove only a tiny fraction of the edges at each
vertex and still make the neighbourhoods of $\Omega(p^{-2})$ vertices into
independent sets. This prompts the question how we should restrict the
adversary so that any~$H$ with small maximum degree and sublinear bandwidth
is contained in~$G$? The following theorem, which is the main result of
this paper, answers this question.

\begin{theorem}[Main result]
\label{thm:main}
For each $\gamma >0$, $\Delta \geq 2$, $k \geq 2$ and $0\le s\le k-1$,
there exist constants $\beta^\ast >0$ and $\Ca >0$ such that the
following holds asymptotically almost surely for $\Gamma = G(n,p)$ if 
$p \geq \Ca\big(\frac{\log n}{n}\big)^{1/\Delta}$. Let $G$ be a spanning subgraph of $\Gamma$ with $\delta(G) \geq\left(\frac{k-1}{k}+ \gamma\right)pn$, such that for each $v\in V(G)$ there are at least $\gamma p^{\binom{s}{2}}(pn)^s$ copies of $K_s$ in $N_G(v)$. Let $H$ be a graph on $n$ vertices with $\Delta(H) \leq \Delta$, bandwidth at most $\beta^\ast n$ and suppose that there is a proper $k$-colouring of $V(H)$ and at least $\Ca p^{-2}$ vertices in $V(H)$ whose neighbourhood contains only $s$ colours. Then $G$ contains a copy of $H$.
\end{theorem}

To help understand this statement, observe that the extra condition we put on $G$ is that each vertex neighbourhood contains a constant (but perhaps rather small) fraction of the copies of $K_s$ which it has in $\Gamma$. The additional restriction on $H$ which is not present in the Bandwidth Theorem specialises to requiring $\Omega(p^{-2})$ vertices not in triangles if $s=1$ (re-proving Theorem~\ref{thm:abet}) and becomes trivially satisfied when $s=k-1$, since no vertex in a proper $k$-colouring can have neighbours of $k$ or more different colours. In particular, imposing the additional requirement on $G$ that every vertex neighbourhood contains many copies of $K_{k-1}$ allows us to find copies of all the graphs handled by the Bandwidth Theorem.

The reader might have expected, by analogy with Theorem~\ref{thm:abet}, to see a different condition on $H$, namely that there should exist $\Omega(p^{-2})$ vertices whose neighbourhood is $s$-colourable (ignoring the overall colouring of $H$). However this condition is not sufficient: we will give in Section~\ref{sec:construct} an example of a graph $H$ which satisfies this condition (for $s=2$) but which need not be a subgraph of $G$ satisfying the conditions of Theorem~\ref{thm:main}.

We should comment on the relation between this result and the recent work of Fischer, \v{S}kori\'c, Steger and Truji\'c~\cite{FSST}, who show `triangle-resilience' for the square of a Hamilton cycle. Triangle-resilience is a stronger condition to impose on $G$ than our Theorem~\ref{thm:main} would require for proving the existence of the square of a Hamilton cycle, so in this sense our result is stronger. However we can only work with $p\gg\big(\tfrac{\log n}{n}\big)^{1/4}$, whereas in~\cite{FSST} $p$ may be as small as $Cn^{-1/2}\log^3n$. This is rather close to the lower bound $p=n^{-1/2}$ at which point even a typical $G(n,p)$ does not contain the square of a Hamilton cycle, so in this sense the result of~\cite{FSST} is much stronger. It would be very interesting to improve the probability bounds in our result. But the method of~\cite{FSST} uses the structure of the square of a Hamilton cycle in an essential way (in particular that it has constant bandwidth), and it is not clear how one might use their ideas in our more general situation.

\subsection{Outline of the paper}
We prove Theorem~\ref{thm:main} by making use of the sparse regularity
lemma of Kohayakawa and R\"odl~\cite{kohayakawa1997,kohayakawa2003}, the
sparse blow-up lemma of~\cite{blowup}, and several lemmas
from~\cite{ABET}. In Section~\ref{sec:preliminaries} we give the
definitions and results necessary to state and use the sparse regularity
lemma and the sparse blow-up lemma, and also a few probabilistic
lemmas. In Section~\ref{sec:mainlemmas} we give a somewhat more
general statement (Theorem~\ref{thm:maink}) than Theorem~\ref{thm:main},
which allows for graphs $H$ which are not quite $k$-colourable, and outline
briefly how to prove it using various lemmas. 

The basic proof strategy, and most
of the lemmas, are taken from~\cite{ABET}. The main exception is the
pre-embedding lemma,
Lemma~\ref{lem:coverv}, which replaces the `Common Neighbourhood Lemma'
of~\cite{ABET}. The proof of this lemma, which is provided in Section~\ref{sec:pel},
 requires new ideas and is the main work of this paper. The setup that this
 pre-embedding lemma creates also entails a number of modifications to the
 proof from~\cite{ABET}, which need some work. The details are given in
 Section~\ref{sec:mainproof}, where we give the proof of the main technical theorem,
Theorem~\ref{thm:maink}.

Finally, we finish with some  concluding remarks in Section~\ref{sec:remarks}.

\section{Preliminaries}
\label{sec:preliminaries}

Throughout the paper $\log$ denotes the natural logarithm.
We assume that the order $n$ of all graphs tends to infinity and therefore is sufficiently large whenever necessary.
Our graph-theoretic notation is standard.
In particular, given a graph $G$ its vertex set
is denoted by $V(G)$ and its edge set by $E(G)$. Let $A,B\subseteq V$ be disjoint vertex sets. We denote the number of edges between $A$ and $B$ by $e(A,B)$.
For a vertex $v \in V(G)$ we write $N_G(v)$ for the neighbourhood of $v$ in $G$ and $N_G(v,A):= N_G(v) \cap A$ for the neighbourhood of $v$ restricted to $A$.
Finally, let $\deg_G(v) := |N_G(v)|$ be the degree of $v$ in $G$.
For the sake of readability, we do not make any effort to optimise the constants in our theorems and proofs.

\subsection{The sparse regularity method}

Now we introduce some definitions and results of the regularity method as well as related tools that are essential in our proofs. In particular, we state a minimum degree version of the sparse regularity lemma (Lemma~\ref{lem:regularitylemma}) and the sparse blow-up lemma (Lemma~\ref{thm:blowup}). Both lemmas use the concept of regular pairs. Let $G= (V,E)$ be a graph, $\eps, d >0$, and $p \in (0,1]$. Moreover, let $X,Y \subseteq V$ be two disjoint nonempty sets. The \emph{$p$-density} of the pair $(X,Y)$ is defined as \[d_{G,p}(X,Y) := \frac{e_G(X,Y)}{p|X||Y|}.\] 
We now define regular, and super-regular, pairs. Note that what we are calling `regular' is sometimes referred to as `lower-regular' by contrast with `fully-regular' (sometimes just called `regular') pairs in which an upper bound on $p$-densities is also imposed. It is immediate from the definition of the latter that a fully-regular pair is also lower-regular, with the same parameters; the converse is false.

\begin{definition}[regular pairs, fully-regular pairs, super-regular pairs]
\label{def:regular}
The pair $(X,Y)$ is called \emph{$(\eps,d,p)_G$-regular} if for every $X'\subseteq X$ and $Y'\subseteq Y$ with $|X'|\geq \eps|X|$ and $|Y'|\geq \eps |Y|$ we have  $d_{G,p}(X',Y') \geq d- \eps$. It is called \emph{$(\eps,d,p)_G$-regular} if there is some $d'\ge d$ such that for every $X'\subseteq X$ and $Y'\subseteq Y$ with $|X'|\geq \eps|X|$ and $|Y'|\geq \eps |Y|$ we have  $\big|d_{G,p}(X',Y')-d'\big| \le \eps$.

If $(X,Y)$ is $(\eps,d,p)_G$-regular, and in addition we have
\begin{align*}
 |N_G(x,Y)| &\geq (d-\eps)\max\big(p|Y|,\deg_\Gamma(x,Y)/2\big)\quad\text{and}\\
 |N_G(y,X)| &\geq (d-\eps)\max\big(p|X|,\deg_\Gamma(y,X)/2\big)
\end{align*}
for every $x \in X$ and $y \in Y$, then the pair $(X,Y)$ is called \emph{$(\eps,d,p)_G$-super-regular}.
\end{definition}

A direct consequence of the definition of $(\eps,d,p)$-regular pairs is the following proposition about the sizes of neighbourhoods in regular pairs. 

\begin{proposition}
\label{prop:neighbourhood}
Let $(X,Y)$ be $(\eps, d,p)$-regular. Then there are less than $\eps |X|$ vertices $x\in X$ with $|N(x,Y)| < (d-\eps)p|Y|$. \qed
\end{proposition} 

The following proposition is another immediate consequence of Definition~\ref{def:regular}. It states that an $(\eps,d,p)$-regular pair is still regular if only a linear fraction of its vertices is removed.

\begin{proposition}
\label{prop:subpairs}
Let $(X,Y)$ be $(\eps, d,p)$-regular and suppose $X'\subseteq X$ and $Y' \subseteq Y$ satisfy $|X'| \geq \mu |X|$ and $|Y'|\geq \nu |Y|$ with some $\mu, \nu >0$. Then $(X',Y')$ is $(\frac{\eps}{\min\{\mu,\nu\}},d,p)$-regular. \qed
\end{proposition}

In order to state the sparse regularity lemma, we need some more definitions. A partition $\cV = \{V_i\}_{i\in\{0,\ldots,r\}}$ of the vertex set of $G$ is called an \emph{$(\eps,p)_G$-regular partition} of $V(G)$ if $|V_0|\leq \eps |V(G)|$ and $(V_i,V_{i'})$ forms an $(\eps,0,p)_G$-fully-regular pair for all but at most $\eps\binom{r}{2}$ pairs $\{i,i'\}\in \binom{[r]}{2}$. It is called an \emph{equipartition} if $|V_i| = |V_{i'}|$ for every $i,i'\in[r]$.
The partition $\cV$ (or the pair $(G,\cV)$) is called
\emph{$(\eps,d,p)_G$-regular} on a graph $R$ with vertex set $[r]$
if $(V_i, V_{i'})$ is $(\eps,d,p)_G$-regular for every $\{i,i'\}
\in E(R)$. The graph $R$ is referred to as the \emph{$(\eps,d,p)_G$-reduced
  graph} of $\cV$, the partition classes $V_i$ with $i \in [r]$ as
\emph{clusters}, and $V_0$ as the \emph{exceptional set}. We also say that
$\cV$ (or the pair $(G,\cV)$) is \emph{$(\eps,d,p)_G$-super-regular} on a graph $R'$ with vertex set $[r]$ if $(V_i, V_{i'})$ is $(\eps,d,p)_G$-super-regular for every $\{i,i'\}\in E(R')$. 

Analogously to Szemeredi's regularity lemma for dense graphs, the sparse
regularity lemma, proved by Kohayakawa, Rödl, and
Scott~\cite{kohayakawa1997, kohayakawa2003, scott2011}, asserts the existence of an $(\eps,p)$-regular partition of constant size of any sparse graph. We state a minimum degree version of this lemma, whose proof can be found in the appendix of~\cite{ABET}.

\begin{lemma}[Minimum degree version of the sparse regularity lemma]
\label{lem:regularitylemma}
For each $\eps >0$, each  $\alpha \in [0,1]$, and $r_0\geq 1$ there exists
$r_1\geq 1$ with the following property. For any $d\in[0,1]$, any $p>0$,
and any $n$-vertex graph $G$ with minimum degree $\alpha p n$ such that for
any disjoint $X,Y\subset V(G)$ with $|X|,|Y|\ge\tfrac{\eps n}{r_1}$ we have
$e(X,Y)\le \big(1+\tfrac{1}{1000}\eps^2\big)p|X||Y|$, there is an
$(\eps,p)_G$-regular equipartition of $V(G)$ with $(\eps,d,p)_G$-reduced
graph $R$ satisfying $\delta(R) \geq (\alpha-d-\eps)|V(R)|$ and $r_0 \leq
|V(R)| \leq r_1$. \qed
\end{lemma}

We will need the following version of the sparse regularity lemma (see e.g.~\cite[Lemma~29]{ABET} for a proof), allowing for a partition equitably refining an initial partition with parts of very different sizes. Given a partition $V(G)=V_1\dcup\dots\dcup V_s$, we say a partition $\{V_{i,j}\}_{i\in[s],j\in[t]}$ is an \emph{equitable $(\eps,p)$-regular refinement} of $\{V_i\}_{i\in[s]}$ if $|V_{i,j}|=|V_{i,j'}|\pm 1$ for each $i\in[s]$ and $j,j'\in[t]$, and there are at most $\eps s^2t^2$ pairs $(V_{i,j},V_{i',j'})$ which are not $(\eps,0,p)$-fully-regular. 

\begin{lemma}[Refining version of the sparse regularity lemma]
\label{lem:SRLb}
For each $\eps>0$ and $s\in\mathbb{N}$ there exists $t_1\geq 1$ such that
the following holds. Given any graph $G$, suppose $V_1\dcup\dots\dcup V_s$
is a partition of $V(G)$. Suppose that $e(V_i)\le 3p|V_i|^2$ for each
$i\in[s]$, and $e(V_i,V_{i'})\le 2p|V_i||V_{i'}|$ for each $i\neq
i'\in[s]$. Then there exist sets $V_{i,0}\subset V_i$ for each $i\in[s]$
with $|V_{i,0}|<\eps|V_i|$, and an equitable $(\eps,p)$-regular refinement
$\{V_{i,j}\}_{i\in[s],j\in[t]}$ of $\{V_i\setminus V_{i,0}\}_{i\in[s]}$ for
some $t\le t_1$. \qed
\end{lemma} 

A key ingredient in the proof of our main theorem is the so-called sparse
blow-up lemma established in~\cite{blowup}. Given a subgraph $G \subseteq \Gamma =G(n,p)$ with $p \gg (\log n/n)^{1/\Delta}$ and an $n$-vertex graph $H$ with maximum degree at most $\Delta$ with vertex partitions $\cV$ and $\cW$, respectively, the sparse blow-up lemma guarantees under certain conditions a spanning embedding of $H$ in $G$ which respects the given partitions. In order to state this lemma we need some definitions.

Let $G$ and $H$ be graphs on $n$ vertices with partitions
$\cV=\{V_i\}_{i\in[r]}$ of $V(G)$ and $\cW=\{W_i\}_{i\in[r]}$ of $V(H)$. We
say that $\cV$ and $\cW$ are \emph{size-compatible} if  $|V_i|=|W_i|$ for
all $i\in[r]$. If there exists an integer $m \geq 1$ such that $m \leq
|V_i| \leq \kappa m$ for every $i\in [r]$, then we say that $(G,\cV)$ is
\emph{$\kappa$-balanced}. Given a graph $R$ on $r$ vertices, we call $(H,
\cW)$ an \emph{$R$-partition} if for every edge $\{x,y\}\in E(H)$ with $x
\in W_i$ and $y\in W_{i'}$ we have $\{i,i'\}\in E(R)$. The following
definition allows for image restrictions in the sparse blow-up lemma.

\begin{definition}[Restriction pair]
\label{def:restrict} 
 Let $\eps,d>0$, $p \in [0,1]$, and let $R$ be a graph on $r$ vertices. Furthermore, let $G$ be a (not necessarily spanning) subgraph of $\Gamma = G(n,p)$ and let $H$ be a graph given with vertex partitions $\cV= \{V_i\}_{i\in[r]}$ and $\cW = \{W_i\}_{i\in[r]}$, respectively, such that $(G,\cV)$ and $(H,\cW)$ are size-compatible $R$-partitions.
  Let $\cI=\{I_x\}_{x\in V(H)}$ be a collection of subsets of $V(G)$, called
  \emph{image restrictions}, and $\cJ=\{J_x\}_{x\in V(H)}$ be a collection of
  subsets of $V(\Gamma)\setminus V(G)$, called \emph{restricting vertices}.
   For each $i\in [r]$ we define $R_i\subseteq W_i$ to be the set of all vertices $x \in W_i$ for which $I_x \neq V_i$. 
  We say that $\cI$ and $\cJ$ are a
  \emph{$(\rho,\zeta,\Delta,\Delta_J)$-restriction pair} if the
  following properties hold for each $i\in[r]$ and $x\in W_i$.
  \begin{enumerate}[label=\itmarab{RP}]
    \item\label{itm:restrict:numres} We have $|R_i|\leq\rho|W_i|$.
    \item\label{itm:restrict:sizeIx} If $x\in R_i$, then $I_x\subseteq
    \bigcap_{u\in J_x} N_\Gamma(u, V_i)$ is of size at least $\zeta(dp)^{|J_x|}|V_i|$.
    \item\label{itm:restrict:Jx} If $x\in R_i$, then $|J_x|+\deg_H(x)\leq\Delta$ and 
    if $x\in W_i\setminus R_i$, then $J_x=\varnothing$.
    \item\label{itm:restrict:DJ} Each vertex in $V(G)$ appears in at most $\Delta_J$ of the sets of $\cJ$.
    \item\label{itm:restrict:sizeGa} We have
    $\big|\bigcap_{u\in J_x} N_\Gamma(u, V_i)\big| = (p\pm\eps p)^{|J_x|}|V_i|$.
    \item\label{itm:restrict:Ireg} If $x\in R_i$, for each $xy\in E(H)$ with $y\in W_j$, 
    \[\text{the pair }\quad\Big( V_i \cap \bigcap_{u\in J_x}N_\Gamma(u), V_j \cap \bigcap_{v\in J_y}N_\Gamma(v)\Big)\quad\text{ is
    $(\eps,d,p)_G$-regular.}\] 
  \end{enumerate}
\end{definition}

The sparse blow-up lemma needs not all pairs in the reduced
graph~$R$ to be super-regular, but only those in a subgraph~$R'$
of~$R$. This, however, is only possible if a good proportion of~$H$ is
embedded to the pairs in~$R'$.
The following definition of buffer-sets makes this requirement
precise. Moreover, we need certain regularity inheritance properties for
the pairs in~$R'$.

\begin{definition}[$(\vartheta, R')$-buffer, regularity inheritance]
\label{def:buffer}
Let $R$ and $R'$ be graphs on vertex set $[r]$ with $R'\subset R$.
Suppose that $(H,\cW)$ is an $R$-partition and that
$(G,\cV)$ is a size-compatible $(\eps,d,p)_G$-regular partition with reduced
graph~$R$. We say
that the family $\tcW=\{\tW_i\}_{i\in[r]}$ of subsets $\tW_i\subseteq W_i$ is an \emph{$(\vartheta,R')$-buffer} for $H$ if
 \begin{enumerate}[label=\rom]
  \item $|\tW_i|\geq\vartheta |W_i|$ for all $i\in[r]$,  and 
  \item for each $i\in[r]$ and each $x\in\tW_i$, the first and second neighbourhood of $x$ go along $R'$, i.e.,\
  for each $\{x,y\},\{y,z\}\in E(H)$ with $y\in W_j$ and $z\in W_k$ we have $\{i,j\}\in E(R')$ and $\{j,k\}\in E(R')$.
 \end{enumerate}
 We say $(G,\cV)$ has \emph{one-sided inheritance} on $R'$ if for
every $\{i,j\}, \{j,k\}\in E(R')$ and every $v\in V_i$ the pair
$\big(N_\Gamma(v, V_j),V_k\big)$ is $(\eps,d,p)_G$-regular.
We say $(G,\cV)$ has \emph{two-sided inheritance} on $R'$ for $\tcW$ if
for each $i,j,k \in V(R')$ such that there is a
triangle $x_ix_jx_k$ in~$H$ with $x_i\in \tW_i$, $x_j\in W_j$, and
   $x_k\in W_k$ the following holds.
For every $v\in V_i$ the pair $\big(N_\Gamma(v, V_j),N_\Gamma(v, V_k)\big)$ is
$(\eps,d,p)_G$-regular.
\end{definition}

Now we can finally state the sparse blow-up lemma.

\begin{lemma}[{Sparse blow-up lemma~\cite[Lemma 1.21]{blowup}}]
\label{thm:blowup}
  For each $\Delta$, $\Delta_{R'}$, $\Delta_J$, $\vartheta,\zeta, d>0$, $\kappa>1$
  there exist $\eBL,\rho>0$ such that for all $r_1$ there is a $\CBL$ such that for
  $p\geq\CBL(\log n/n)^{1/\Delta}$ the random graph $\Gamma=G_{n,p}$ asymptotically
  almost surely satisfies the following.
   
  Let $R$ be a graph on $r\le r_1$ vertices and let $R'\subseteq R$ be a spanning
  subgraph with $\Delta(R')\leq \Delta_{R'}$.
  Let $H$ and $G\subseteq \Gamma$ be graphs given with $\kappa$-balanced,
  size-compatible vertex partitions 
  $\cW=\{W_i\}_{i\in[r]}$ and $\cV=\{V_i\}_{i\in[r]}$ with parts of size at
  least $m\geq n/(\kappa r_1)$. 
  Let $\cI=\{I_x\}_{x\in V(H)}$ be a family of image restrictions, and
  $\cJ=\{J_x\}_{x\in  V(H)}$  be a family of restricting vertices.
  Suppose that
  \begin{enumerate}[label=\itmarab{BUL}]
  \item\label{itm:blowup:H} $\Delta(H)\leq \Delta$, 
	for every edge $\{x,y\}\in E(H)$ with $x\in W_i$ and $y\in W_j$ we have $\{i,j\}\in E(R)$ and $\tcW=\{\tW_i\}_{i\in[r]}$ is an
    $(\vartheta,R')$-buffer for $H$,
	\item\label{itm:blowup:G} $(G,\cV)$ is $(\eBL,d,p)_G$-regular on $R$, $(\eBL,d,p)_G$-super-regular on $R'$, has one-sided inheritance on $R'$, and two-sided inheritance on $R'$ for $\tcW$,
  \item\label{itm:blowup:restrict} $\cI$ and $\cJ$ form
    a $(\rho,\zeta,\Delta,\Delta_J)$-restriction pair.
  \end{enumerate}
  Then there is an embedding $\phi\colon V(H)\to V(G)$ such that $\phi(x)\in
  I_x$ for each $x\in H$. \qed
\end{lemma}	

Observe that in the blow-up lemma for dense graphs, proved by Koml{\'o}s, S{\'a}rk{\"o}zy, and Szemer{\'e}di~\cite{komlos1997blow}, one does not need to explicitly ask for one- and two-sided inheritance properties since they are always fulfilled by dense regular partitions. This is, however, not true in general in the sparse setting. The following two lemmas will be very useful whenever we need to redistribute vertex partitions in order to achieve some regularity inheritance properties.

\begin{lemma}[One-sided regularity inheritance~\cite{blowup}]
\label{lem:OSRIL}
For each $\eo, \ao >0$ there exist $\eps_0 >0$ and $C >0$ such that for any
$0 < \eps < \eps_0$ and $0 < p <1$  asymptotically almost surely $\Gamma=
G(n,p)$ has the following property. For any disjoint sets $X$ and $Y$ in
$V(\Gamma)$ with $|X|\geq C\max\big(p^{-2}, p^{-1} \log n\big)$ and $|Y|
\geq C p^{-1} \log n$, and any subgraph $G$ of $\Gamma[X,Y]$ which is
$(\eps, \ao,p)_G$-regular, there are at most $C p^{-1}\log (en/|X|)$
vertices $z \in V(\Gamma)$ such that $(X \cap N_{\Gamma}(z),Y)$ is not
$(\eo,\ao,p)_G$-regular. \qed
\end{lemma}

\begin{lemma}[Two-sided regularity inheritance~\cite{blowup}]
\label{lem:TSRIL}
For each $\et,\at>0$ there exist $\eps_0>0$ and
$C >0$ such that for any $0<\eps<\eps_0$ and $0<p<1$, asymptotically almost surely
$\Gamma=G_{n,p}$ has the following property. For any disjoint sets $X$
and $Y$ in $V(\Gamma)$ with $|X|,|Y|\ge C\max\{p^{-2},p^{-1}\log n\}$, and any
subgraph $G$ of $\Gamma[X,Y]$ which is $(\eps,\at,p)_G$-regular, there are
at most $C\max\{p^{-2},p^{-1}\log (en/|X|)\}$ vertices $z \in V(\Gamma)$
such that $\big(X\cap N_\Gamma(z),Y\cap N_\Gamma(z)\big)$ is not
$(\et,\at,p)_G$-regular. \qed
\end{lemma}

Finally, we need a statement about random subpairs of regular pairs (which is used to prove Lemma~\ref{lem:TSRIL}).

\begin{corollary}[{\cite[Corollary 3.8]{GKRS}}]\label{cor:TSI}
    For any $d$, $\beta$, $\eps'>0$ there exist
    $\eps_0>0$ and $C$ such that for
    any $0<\eps<\eps_0$ and $0<p<1$, if $(X,Y)$ is an $(\eps,d,p)$-regular
    pair in a graph $G$, then the number of pairs $X'\subseteq X$ and
    $Y'\subseteq Y$ with $|X'|=w_1\ge C/p$ and $|Y'|=w_2\ge C/p$ such that
    $(X',Y')$ is an $(\eps',d,p)$-regular pair in~$G$ is at least
    $(1-\beta^{\min(w_1,w_2)})\binom{|X|}{w_1}\binom{|Y|}{w_2}$. \qed
\end{corollary}

\subsection{Concentration inequalities}

We close this section with two of Chernoff's bounds for random variables
that follow a binomial (Theorem~\ref{thm:chernoff}) and a hypergeometric
distribution (Theorem~\ref{thm:hypergeometric}), respectively, and the
following useful observation. Roughly speaking, it states that
a.a.s.~nearly all vertices in $G(n,p)$ have approximately the expected
number of neighbours within large enough subsets (for a proof see e.g.~\cite[Proposition~18]{ABET}).

\begin{proposition}
\label{prop:chernoff}
For each $\eps>0$ there exists a constant $C >0$ such that for every
$0<p<1$ asymptotically almost surely $\Gamma=G(n,p)$ has the 
property that for any sets $X,Y\subset V(\Gamma)$ with $|X|\ge
Cp^{-1}\log n$ and $|Y|\ge Cp^{-1}\log (en/|X|)$ the following holds.
\begin{enumerate}[label=\abc]
\item If~$X$ and~$Y$ are disjoint, then $e(X,Y)=(1\pm\eps)p|X||Y|$. 
\item We have $e(X)\le 2p|X|^2$. 
\item At most $C p^{-1} \log (en/|X|)$ vertices $v \in V(\Gamma)$ satisfy $\big||\NGa(v,X)| - p
  |X|\big| > \eps p |X|$.
\end{enumerate}\qed
\end{proposition}


We use the following version of Chernoff's Inequalities (see e.g.~\cite[Chapter~2]{janson2011random} for a proof).

\begin{theorem}[Chernoff's Inequality,~\cite{janson2011random}]
\label{thm:chernoff}
Let $X$ be a random variable which is the sum of independent Bernoulli random variables. Then we have for $\eps\leq 3/2$
\[\Pr\big[|X-\Ex[X]| > \eps \Ex[X]\big] < 2e^{-\eps^2\Ex[X]/3}\,.\]
Furthermore, if $t\ge 6\Ex[X]$ then we have
\[\Pr\big[X\ge\Ex[X]+t\big]\le e^{-t}\,.\] \qed
\end{theorem}

Finally, let $N$, $m$, and $s$ be positive integers and let $S$ and $S' \subseteq S$ be two sets with $|S| = N$ and $|S'| = m$. The \emph{hypergeometric distribution} is the distribution of the random variable $X$ that is defined by drawing $s$ elements of $S$ without replacement and counting how many of them belong to $S'$. It can be shown that Theorem~\ref{thm:chernoff} still holds in the case of hypergeometric distributions (see e.g.~\cite{janson2011random}, Chapter~2 for a proof) with $\Ex[X]= ms/N$.

\begin{theorem}[Hypergeometric inequality,~\cite{janson2011random}]
\label{thm:hypergeometric}
Let $X$ be a random variable is hypergeometrically distributed with parameters $N$, $m$, and $s$. Then for any $\eps>0$ and $t\ge\eps ms/N$ we have
\[\Pr\big[|X - ms/N| > t \big] < 2e^{-\eps^2t/3}\,.\] \qed
\end{theorem}

We require the following technical lemma, which is a consequence of the hypergeometric inequality stated in Theorem~\ref{thm:hypergeometric}. 

\begin{lemma}\label{lem:hypgeo}
 For each $\eps^+_0,d^+>0$ there exists $\eps^+>0$, and for each
 $\eps,d>0$ there exists $\eps^->0,$ such that for each
 $\eta>0$ and $\Delta$ there exists $C$ such that the following holds for
 each $p>0$. 

Let $W\subset [n]$, let $t\le 100n^{\Delta+1}$, and let
 $T_1,\ldots,T_t$ be subsets of $W$. Let~$G$ be a graph on~$W$.
 For each $i\in[t]$ let $(X_i,Y_i)$ be a pair
 which is either $(\eps^+,d^+,p)_G$-regular, or
 $(\eps^-,d,p)_G$-regular (respectively), and which satisfies
$m|X_i|/|W|,m|Y_i|/|W|\ge 2Cp^{-1}\log n$.

For each $m\le |W|$ there is a set $S\subset W$ of size $m$ such that for each $i\in[t]$
 \[|T_i\cap S|=\tfrac{m}{|W|}|T_i|\pm \big(\eta|T_i|+C\log n\big)\,,\]
and the pair $\big(X_i\cap S,Y_i\cap S\big)$
 is $\big(\eps^+_0,d^+,p\big)$-regular, or
 $(\eps,d,p)$-regular (respectively).
\end{lemma}
\begin{proof}
 Given $\eps^+_0,d^+$, let $\eps^+$ be returned by Corollary~\ref{cor:TSI}
 for input $d^+$, $\beta=\tfrac12$ and $\eps^+_0$. 
 Given $\eps,d$, let $\eps^-$ be returned by Corollary~\ref{cor:TSI} for input $d$, $\beta=\tfrac12$ and $\eps$.
 Let $C\ge 30\eta^{-2}\Delta$ be large enough for these applications of Corollary~\ref{cor:TSI}.
 
 Observe that for each $i$, the size of $T_i\cap S$ is hypergeometrically distributed.
 By Theorem~\ref{thm:hypergeometric}, for each $i$ we have
 \[\Pr\big[|T_i\cap S|\neq \tfrac{m}{|W|}|T_i|\pm \big(\eta|T_i|+C\log n\big)\big]<2e^{-\eta^2C\log n/3}<\frac{2}{n^{2+\Delta}}\,,\]
 so taking the union bound over all $i\in[t]$ we conclude that the probability of failure is at most $2t/n^{2+\Delta}\le 200/n\to 0$ as $n\to\infty$, as desired.
 
 To obtain the second property, observe that 
 Theorem~\ref{thm:hypergeometric} also implies that we have $|X_i\cap
 S|,|Y_i\cap S|\ge Cp^{-1}\log n$ for each $i\in[t]$ with probability
 tending to one as $n\to\infty$. Conditioning on the size of $|X_i\cap S|$,
 the set $X_i\cap S$ is a uniformly distributed subset of $X_i$ of size
 $|X_i\cap S|$, and the same applies to $Y_i\cap S$. Now
 Corollary~\ref{cor:TSI} says that, conditioning on $|X_i\cap S|,|Y_i\cap
 S|\ge Cp^{-1}\log n$, the probability that $\big(X_i\cap S,Y_i\cap S\big)$
 fails to have the desired regularity in $G$ is at most $2^{-Cp^{-1}\log n}$, and taking a union bound over the choices of~$i$ the result follows.
\end{proof}




\section{Main technical result and main lemmas}
\label{sec:mainlemmas}

We deduce Theorem~\ref{thm:main} from the following technical result
(corresponding results also appear in the predecessor papers~\cite{ABET,bottcher2009proof}). This result is more general in that it allows for an extra colour, \emph{zero}, in the colouring of $H$, provided that this colour does not appear too often.

\begin{definition}[Zero-free colouring]\label{def:zerofree}
 Let $H$ be a $(k+1)$-colourable graph on $n$ vertices and let $\cL$ be a labelling of its vertex set of bandwidth at most $\beta n$. A proper $(k+1)$- colouring $\sigma:V(H) \to \{0,\ldots,k\}$ of its vertex set is said to be \emph{$(z,\beta)$-zero-free} with respect to $\cL$ if any $z$ consecutive blocks contain at most one block with colour zero, where a block is defined as a set of the form $\{(t-1)4k\beta n +1, \ldots, t4k\beta n\}$ with $t \in [1/(4k\beta)]$.
\end{definition}

\begin{theorem}[Main technical result]
\label{thm:maink}
For each $\gamma>0$, $\Delta \geq 2$, $k\geq 2$ and $1\le s\le k-1$, there exist constants $\beta >0$, $z>0$, and $C>0$ such that the following holds asymptotically almost surely for $\Gamma = G(n,p)$ if $p\geq C\big(\frac{\log n}{n}\big)^{1/\Delta}$. Let $G$ be a spanning subgraph of $\Gamma$ with $\delta(G) \geq\big(\frac{k-1}{k}+\gamma\big)pn$   such that for each $v\in V(G)$ there are at least $\gamma p^{\binom{s}{2}}(pn)^s$ copies of $K_s$ in $N_G(v)$ and let $H$ be a graph on $n$ vertices with $\Delta(H) \leq \Delta$ that has a labelling $\cL$ of its vertex set of bandwidth at most $\beta n$, a $(k+1)$-colouring that is $(z,\beta)$-zero-free with respect to $\cL$ and where the first $\sqrt{\beta} n$ vertices in $\cL$ are not given colour zero and the first $\beta n$ vertices in $\cL$ include $Cp^{-2}$ vertices whose neighbourhood contains only $s$ colours. Then $G$ contains a copy of $H$.
\end{theorem}

The basic proof strategy for this theorem is analogous to the proof strategy 
for~\cite[Theorem~23]{ABET}. Eventually, we will apply the sparse blow-up lemma,
Lemma~\ref{thm:blowup}, to embed most of~$H$ into $G$, and we need to obtain the
necessary conditions for this lemma. The difficulty is that, whatever
regular partition of $G$ we take, there may be some exceptional vertices
which are `badly behaved' with respect to this partition. Our first main
lemma, the following Lemma for~$G$, states that there is a partition with
only few such vertices, which we collect in a set~$V_0$. These vertices will be dealt with in a
pre-embedding stage before the application of the sparse blow-up lemma.

For the application of the sparse blow-up lemma the following two graphs
$B^k_r$ and $K^k_r$,
which we shall find as subgraphs of the reduced graph of~$G$, are essential.
Let $r, k \geq 1$ and let $B^k_r$ be the \emph{backbone graph} on $kr$ vertices. That is, we have
\[V(B^k_r) := [r] \times [k]\] and for every $j \neq j' \in [k]$ we have $\{(i,j),(i',j')\} \in E(B^k_r)$ if and only if $|i-i'|\le1$.
Let $K^k_r \subseteq B^k_r$ be the spanning subgraph of $B^k_r$ that is the disjoint union of $r$ complete graphs on $k$ vertices given by the following components: the complete graph $K^k_r[\{(i,1),\ldots, (i,k)\}]$ is called the \emph{$i$-th component} of $K^k_r$ for each $i\in [r]$. 

A vertex partition $\cV' = \{V_{i,j}\}_{i\in[r],j\in[k]}$ is called \emph{$k$-equitable} if $\big||V_{i,j}| -|V_{i,j'}|\big|\leq 1$ for every $i\in [r]$ and $j,j'\in[k]$. Similarly, an integer partition $\{n_{i,j}\}_{i\in[r],j\in[k]}$ of $n$ (meaning that $n_{i,j} \in \mathbb Z_{\geq 0}$ for every $i\in [r],j\in[k]$ and $\sum_{i\in[r]j\in[k]} n_{i,j} = n$) is \emph{$k$-equitable} if $|n_{i,j}-n_{i,j'}| \leq 1$ for every $i\in[r]$ and $j,j'\in[k]$. 

The Lemma for~$G$ then guarantees a $k$-equitable partition for~$G$ whose
reduced graph~$R^k_r$ contains a copy of the backbone graph~$B^k_r$, is
super-regular on $K^k_r\subset B^k_r$, and satisfies certain regularity
inheritance properties.

\begin{lemma}[Lemma for $G$,~{\cite[Lemma~24]{ABET}}] 
\label{lem:G}
For each $\gamma > 0$ and integers $k \geq 2$ and $r_0 \geq 1$ there exists $d > 0$ such that for every $\eps \in \left(0, \frac{1}{2k}\right)$ there exist $r_1\geq 1$ and $\Ca>0$ such that the following holds a.a.s.~for $\Gamma = G(n,p)$ if $p \geq \Ca \left(\log n/n\right)^{1/2}$. Let $G=(V,E)$ be a spanning subgraph of $\Gamma$ with $\delta(G) \geq  \left(\frac{k-1}{k} + \gamma\right)pn$. Then there exists an integer $r$ with $r_0\leq kr \leq r_1$, a subset $V_0 \subseteq V$  with $|V_0| \leq \Ca p^{-2}$, 
a $k$-equitable vertex partition $\cV = \{\Vij\}_{i\in[r],j\in[k]}$ of $V(G)\setminus V_0$,
and a graph $R^k_r$ on the vertex set $[r] \times [k]$ with $K^k_r \subseteq B^k_r \subseteq R^k_r$, with $\delta(R^k_r) \geq \left(\frac{k-1}{k} + \frac{\gamma}{2}\right)kr$, and such that the following is true.
\begin{enumerate}[label=\itmarab{G}]
\item \label{lemG:size} $\frac{n}{4kr}\leq |\Vij| \leq \frac{4n}{kr}$ for every $i\in[r]$ and $j\in[k]$,
\item \label{lemG:regular} $\cV$ is $(\eps,d,p)_G$-regular on $R^k_r$ and $(\eps,d,p)_G$-super-regular on $K^k_r$,
\item \label{lemG:inheritance} both $\big(\NGa(v, V_{i,j}),V_{i',j'}\big)$ and $\big(\NGa(v', V_{i,j}),\NGa(v, V_{i',j'})\big)$ are $(\eps,d,p)_G$-regular pairs for every $\{(i,j),(i',j')\} \in E(R^k_r)$ and $v\in V\setminus V_0$,
\item \label{lemG:gamma} $|\NGa(v,V_{i,j})| = (1 \pm \eps)p|V_{i,j}|$ for every $i \in [r]$, $j\in [k]$ and every $v \in V \setminus V_0$.
\end{enumerate}
\end{lemma}

The next step is to find a partition of $H$ which more or less matches that
of $G$. This partition of~$H$ defines an assignment of the vertices of~$H$
to the clusters of~$G$.
In other words, we assign the vertices in $V(H)$ indices $(i,j)$ of the
partition $\cV$, such that about $|V_{i,j}|$ vertices are assigned $(i,j)$
and all edges of~$H$ are assigned to edges of $R^k_r$. In fact, the lemma
states further that most edges of~$H$ are assigned to edges of~$K^k_r$, and
only those incident to vertices of a small set of special vertices~$X$ may be assigned
to other edges of~$R^k_r$.

\begin{lemma}[Lemma for H,~{\cite[Lemma~25]{ABET}}]\label{lem:H2}
Given $D, k, r \geq 1$ and $\xi, \beta > 0 $ the following holds if $\xi \leq 1/(kr)$ and $\beta \leq 10^{-10}\xi^2/(D k^4r)$. Let $H$ be a $D$-degenerate graph on $n$ vertices, let $\mathcal L$ be a labelling of its vertex set of bandwidth at most $\beta n$ and let $\sigma: V(H) \to \{0,\ldots k\}$ be a proper $(k+1)$-colouring that is $(10/\xi, \beta)$-zero-free with respect to $\mathcal L$, where the colour zero does not appear in the first $\sqrt{\beta}n$ vertices of $\mathcal{L}$. Furthermore, let $R^k_r$ be a graph on vertex set $[r] \times [k]$ with $K^k_r \subseteq B^k_r \subseteq R^k_r$ such that for every $i\in [r]$ there exists a vertex $z_i \in \big([r]\setminus\{i\}\big) \times [k]$ with $\big\{z_i, (i,j)\big\} \in E(R^k_r)$ for every $j\in [k]$. Then, given a $k$-equitable integer partition $\{m_{i,j}\}_{i\in[r],j\in[k]}$ of $n$ with $n/(10kr) \leq m_{i,j} \leq 10n/(kr)$ for every $i \in[r]$ and $j\in [k]$, there exists a mapping $f \colon V(H) \to [r]\times[k]$ and a set of special vertices $X \subseteq V(H)$ such that we have for every $i\in [r]$ and $j\in[k]$

\begin{enumerate}[label=\itmarab{H}]
\item\label{lemH:H1} $m_{i,j} - \xi n  \leq |f^{-1}(i,j)| \leq m_{i,j} + \xi n$,
\item\label{lemH:H2} $|X| \leq \xi n$,
\item\label{lemH:H3} $\{f(x),f(y)\} \in E(R^k_r)$  for every $\{x,y\} \in E(H)$,
\item\label{lemH:H4} $y,z\in \cup_{j'\in[k]}f^{-1}(i,j')$ for every $x\in f^{-1}(i,j)\setminus X$ and $xy,yz\in E(H)$, and
\item\label{lemH:H5} $f(x) = \big(1, \sigma(x)\big)$ for every $x$ in the first $\sqrt{\beta}n$ vertices of $\mathcal{L}$.
\end{enumerate}
\end{lemma}

Our next lemma concerns the pre-embedding stage, in which we cover the
vertices in $V_0\subset V(G)$ with vertices of~$H$.  For this purpose we
use the vertices of~$H$ whose neighbourhood contains only~$s$
colours. Let~$x$ be one of these vertices, let~$H'$ be the subgraph
of~$H$ induced on all vertices of distance at most~$s+1$ from~$x$
(including~$x$), and let~$T$ be the set of those vertices in~$H'$ of
distance exactly $s+1$ from~$x$. We cover a vertex~$v$ of $V_0$ by
embedding such a vertex~$x$ of~$H$ onto~$v$, and we also embed all other
vertices in the corresponding~$H'$ which are not in~$T$.  This creates
image restrictions on the vertices of $G$ to which we can embed the
vertices in~$T$.  For the application of Lemma~\ref{thm:blowup} we need
that these image restrictions satisfy certain conditions, and that this
pre-embedding preserves the super-regularity of the remaining partition
of~$G$. For achieving the latter we take a random induced subgraph~$G'$
of~$G$ containing roughly $\mu n$ vertices, and perform the pre-embedding in~$G'$ only. In each cluster of~$G$, the
subgraph~$G'$ selects roughly a $\mu$-fraction of the vertices, and
the induced partitions on~$G'$ and on $G-V(G')$ are also super-regular.
The next lemma states that we can also obtain suitable image restrictions for the vertices in~$T$ while
performing the pre-embedding in~$G'$.

This lemma is a main difference to the proof in~\cite{ABET} and is the
place where we need that the neighbourhood of every vertex in $G$ has a
certain density of $K_s$'s. Another difference to our proof strategy that
this lemma creates, is that it selects a clique $\{\kq_1,\ldots,\kq_k\}$ in~$R$, which might not be one of the cliques of the chosen $K_r^k\subseteq R$, and the vertices of~$T$ are assigned to the corresponding clusters in $G$ (that is, the image
restriction of~$y\in T$ is a subset of the cluster $V_{\kq_j}$ to which it is assigned).
This assignment may well differ from the assignment given
by the Lemma for~$H$, so in our proof of Theorem~\ref{thm:maink} we need to
adapt to this difference by reassigning some more $H$-vertices.

\begin{lemma}[Pre-embedding lemma]\label{lem:coverv}
  For $\Delta,k\ge 2$, $2\le s\le k-1$, and $\gamma,d>0$ with
  $d \le \frac{\gamma}{32}$ there exists $\zeta>0$ such that for every
  $\eps'>0$ there exists $\eps_0>0$ such that for all $0 < \eps < \eps_0$,
  all $\mu>0$ and $r\geq 10^5\gamma^{-1}$, there exists a constant $C^\ast>0$ such that
  the random graph $\Gamma=G(n,p)$ a.a.s.\ has the following property if
  $p\ge C^\ast\big(\tfrac{\log n}{n}\big)^{1/\Delta}$. Suppose we have the
  following setup.

  \begin{enumerate}[label=\itmarab{P}]
    \item $H'$ is a graph with $\Delta(H')\le\Delta$, with a root vertex $x$,
  and no vertex at distance greater than $s+1$ from $x$. 

    \item $\rho$ is a
      proper $k$-colouring of $V(H')$ in which $N(x)$ receives at most $s$
      colours, and $T$ is the set of vertices in $H'$ at distance exactly
      $s+1$ from $x$.  

    \item $G$ is a spanning subgraph of $\Gamma$ with
      $\delta(G)\ge\big(\tfrac{k-1}{k}+{\gamma}\big)pn$ with an
      $(\eps,p)$-regular partition
      $V(G)=V_0\dcup V_1\dcup\dots\dcup V_r$ with $(\eps,d,p)$-reduced
      graph $R$, and such that $\frac{n}{4r} \le |V_i| \le \frac{4n}{r}$
      for all $i \in [r]$.

  \item $G'\subset G$ is a graph with $|V(G')|=(1\pm\eps)\mu n$, with
    $\delta(G')\ge\big(\tfrac{k-1}{k}+{\gamma}\big)p|V(G')|$, and
    $|N_{G'}(W)| \le 2 \mu n p^{t}$ for any set $W \subset V(G')$ of size
    $t \le \Delta$. Suppose further that
    $\left|V_i\cap V(G')\right|=(1\pm\eps)\mu|V_i|$ for each $i$, and that
    $V_0\cap V(G'),\dots,V_r\cap V(G')$ is also an $(\eps,p)$-regular
    partition of $G'$ with $(\eps,d,p)$-reduced graph $R$.

  \item $v\in V(G')$ is a vertex such that there are at least
    $\gamma p^{\binom{s+1}{2}}(\mu n)^s$ copies of $K_s$ in $N_{G'}(v)$ .
  \end{enumerate}
  Then there exist a partial embedding $\phi:V(H')\setminus T\to V(G')$ of $H'$ into $G'$ 
  and a subset $\{\kq_1,\dots,\kq_k\}\subset
  [r]$ with the following properties. For each $u,u'\in T$,
  each $j\in[k]$, and for $\Pi(u)=\phi\big(N_{H'}(u)\cap \Dom(\phi)\big)$, we have
  \begin{enumerate}[label=\itmarabp{P}{'}]
     \item\label{item:pel-v} $\phi(x)=v$.
     \item\label{item:pel-clique} $\kq_1,\dots,\kq_k$ forms a clique in $R$.
     \item\label{item:pel-size} $\big|N_\Gamma\big(\Pi(u)\big)\cap V_{\kq_{\rho(u)}}\big|=(1\pm\eps')p^{|\Pi(u)|}|V_{\kq_{\rho(u)}}|$.
     \item\label{item:pel-minsize} $\big|N_G\big(\Pi(u)\big)\cap V_{\kq_{\rho(u)}}\cap V(G')\big|\ge2\zeta p^{|\Pi(u)|}|V_{\kq_{\rho(u)}}\cap V(G')|$.
     \item\label{item:pel-osril} If $j\neq\rho(u)$ and $|\Pi(u)|\le\Delta-1$ then the pair $\big(N_\Gamma(\Pi(u),V_{\kq_{\rho(u)}}),V_{\kq_j}\big)$ is $(\eps',d,p)_G$-regular.
     \item\label{item:pel-tsril} If $uu'\in H'$ then the pair $\big(N_\Gamma(\Pi(u),V_{\kq_{\rho(u)}}),N_\Gamma(\Pi(u'),V_{\kq_{\rho(u')}})\big)$ is $(\eps',d,p)_G$-regular.
 \end{enumerate}
\end{lemma}

After the pre-embedding stage, we want to apply the sparse blow-up lemma to
embed the remainder of~$H$.
However, the sizes of the clusters $V_{i,j}$ from Lemma~\ref{lem:G} do not quite match the sizes of the sets $X_{i,j}$ from Lemma~\ref{lem:H2}. Also, Lemma~\ref{lem:coverv} embeds some vertices, creating a little further imbalance, and we need to slightly alter the mapping $f$ from Lemma~\ref{lem:H2} to accommodate these pre-embedded vertices. The next lemma allows us to change the sizes of the clusters $V_{i,j}$ slightly to match the partition of $H$, without destroying the properties of the partition of $G$ and of the pre-embedded vertices we worked to achieve.

\begin{lemma}[Balancing lemma,~{\cite[Lemma~27]{ABET}}]
\label{lem:balancing}
For all integers $k\geq 1$, $r_1, \Delta \geq 1$, and reals $\gamma, d >0$ and $0 < \eps < \min\{d,1/(2k)\}$ there exist $\xi >0$ and $\Ca >0$ 
such that the following is true for every $p \geq \Ca \left(\log n/n\right)^{1/2}$ and every $10\gamma^{-1}\le r \leq r_1$ provided that $n$ is large enough. Let~$\Gamma$ be a graph on vertex set $[n]$ and let $G=(V,E)\subseteq \Gamma$ be a (not necessarily spanning) subgraph with vertex partition $\cV = \{\Vij\}_{i\in[r],j\in[k]}$ that satisfies $n/(8kr) \leq |\Vij| \leq 4n/(kr)$ for each $i\in[r]$, $j\in[k]$. 
Let $\{n_{i,j}\}_{i \in [r], j\in [k]}$ be an integer partition of $\sum_{i\in[r],j\in[k]} |V_{i,j}|$. Let $R^k_r$ be a graph on the vertex set $[r] \times [k]$ with minimum degree $\delta(R^k_r) \geq \big((k-1)/k+\gamma/2\big) kr$ such that $K^k_r \subseteq B^k_r \subseteq R^k_r$.
Suppose that the partition $\cV$ satisfies the following properties for
each $i\in[r]$, each $j\neq j'\in[k]$, and each $v\in V$. Suppose we have
\begin{enumerate}[label=\itmarab{B}]
\item \label{lembalancing:sizes} $n_{i,j} -  \xi n \leq |V_{i,j}| \leq n_{i,j} +  \xi n$,
\item \label{lembalancing:regular1} $\cV$ is $\big(\tfrac{\eps}{4},d,p\big)_G$-regular on $R^k_r$ and $\big(\tfrac{\eps}{4},d,p\big)_G$-super-regular on $K^k_r$,
\item \label{lembalancing:inheritance1}  $\big(\NGa(v, V_{i,j}),V_{i,j'}\big)$ and $\big(\NGa(v, V_{i,j}),\NGa(v, V_{i,j'})\big)$ are $\big(\tfrac{\eps}{4},d,p\big)_G$-regular, and
\item \label{lembalancing:gamma1} $|\NGa(v,V_{i,j})| = \big(1 \pm \tfrac{\eps}{4}\big)p|\Vij|$.
\end{enumerate}
Then, there exists a partition $\mathcal{V'}=
\{V'_{i,j}\}_{i\in[r],j\in[k]}$ of $V$ such that for each $i\in[r]$, each
$j\neq j'\in [k]$, and each $v\in V$ we have
\begin{enumerate}[label=\itmarabp{B}{'}]
\item\label{lembalancing:sizesout} $|V'_{i,j}|=n_{i,j}$,
\item\label{lembalancing:symd} $|V_{i,j}\symd V'_{i,j}|\le 10^{-10}\eps^4k^{-2}r_1^{-2} n$,
\item \label{lembalancing:regular} $\mathcal{V'}$ is $(\eps,d,p)_G$-regular on $R^k_r$ and $(\eps,d,p)_G$-super-regular on $K^k_r$,
\item \label{lembalancing:inheritance} $\big(N_{\Gamma}(v,V'_{i,j}), V'_{i,j'}\big)$ and $\big(N_{\Gamma}(v,V'_{i,j}), N_\Gamma(v,V'_{i,j'})\big)$ are $(\eps,d,p)_G$-regular, and
\item\label{lembalancing:gammaout} for each $1\le s\le\Delta$ and for each $v_1,\ldots,v_s\in[n]$
\[\Big|\bigcap_{i\in[s]} N_\Gamma(v_1,V_{i,j})\symd \bigcap_{i\in[s]} N_\Gamma(v_1,V'_{i,j})\Big|\le 10^{-10}\eps^4k^{-2}r_1^{-2}\deg_\Gamma(v_1,\ldots,v_s)+\Ca\log n\,.\]
\end{enumerate} 
\end{lemma} 

After applying Lemma~\ref{lem:balancing} it remains only to check that the
conditions of Lemma~\ref{thm:blowup} are met to complete the embedding of
$H$ and thus the proof of Theorem~\ref{thm:maink}.

\section{Proof of the pre-embedding lemma}
\label{sec:pel}

The basic idea of the proof is as follows. We construct $\phi$ by first embedding $x$ to $v$, then the neighbours of $x$ to $W:=N_{G'}(v)$, and then we keep embedding further vertices at greater distance from $x$ into $G'$ until we finally reach the neighbours of $T$ which we embed such that their neighbourhoods in the regular partition $(V_q)_{q\in R}$ have the desired sizes and regularity properties.

We want to use the regularity method to perform this embedding. That is, we begin by assigning each vertex of $H'$ to a cluster, such that each edge of $H'$ is assigned to two clusters which form an $(\eps,d,p)$-regular pair; these clusters are the initial \emph{candidate sets} for the vertices of $H'$. We embed vertices sequentially. When we embed $y\in V(H')$, we will naturally decrease the candidate sets for $z\in V(H')$ such that $yz\in E(H')$; we need to choose an image for $y$ which has enough $G$-neighbours in these candidate sets in order for the embedding to continue. To guarantee this, we will also need to keep track of $\Gamma$-neighbourhoods, and we will need to maintain the property that if two vertices of $H'$ are adjacent, then their candidate sets form an $(\eps,d,p)$-regular pair. All these properties are easy to maintain because only very few vertices will fail them at any step. The idea is that at the end, the candidate sets for the vertices of $T$ (which we do not embed) are large enough for~\ref{item:pel-minsize}.

In order to apply this strategy to prove Lemma~\ref{lem:coverv}, we need a new \emph{fine} regular partition, which partitions $W$ into clusters and which is also consistent with $(V_q)_{q\in R}$. The obvious way to do this is to apply Lemma~\ref{lem:SRLb} with $W$ and the sets $(V(G')\cap V_q\setminus W)_{q\in V(R)}$ as an initial partition. However this turns out not to work: we need the bound $t_1$ of the number of parts in the fine partition produced by Lemma~\ref{lem:SRLb} to be small compared to $\eps^{-1}$, where $\eps$ is the regularity from Lemma~\ref{lem:G}, but we have $v(R)\gg\eps^{-1}$ and $t_1\gg v(R)$, giving a circular dependence.  So what we do instead is to fix a number $\ell$ which is small compared to $\eps^{-1}$, select a set $L$ of $\ell$ clusters from $(V_q)_{q\in V(R)}$, and apply Lemma~\ref{lem:SRLb} to the union of $W$ and the $(V(G')\cap V_q\setminus W)_{q\in L}$ with that initial partition; this breaks the circular dependence.

We need $R[L]$ to have similar statistical properties to $R$ in order to begin the regularity embedding, i.e.\ to find parts of the fine partition to which we can assign the vertices of $H'$. We show that a random choice of $L$ is likely to give the desired $R[L]$.

\begin{proof}[Proof of Lemma~\ref{lem:coverv}]
First we fix all constants that we need throughout the proof.
Let $\Delta, k\geq 2$ and $\gamma, d>0$ be given. Recall $d\le\tfrac{\gamma}{32}$ by assumption of the lemma.
Let $d' = \min(\tfrac12d, 10^{-5k}\gamma)$ and choose $\xi= 10^{-6}2^{-k}\gamma$ and an integer
\[\ell=\max\big(1000\xi^{-6}\log\xi^{-1},100\cdot 2^k\gamma^{-1}\big)\,.\]

Let $\nua_{\Delta-1,\Delta-1}=\nua_{i,\Delta}=\nua_{\Delta,i}=\tfrac{1}{100\Delta}8^{-\Delta}d'^{\Delta}$ for $i \in [\Delta]$.
For each $(i,j) \in\{0, \dots, \Delta-1\}^2\setminus\{(\Delta-1,\Delta-1)\}$ in reverse lexicographic order, we choose~$\nua_{i,j}\le\nua_{i+1,j},\nua_{i,j+1},\nua_{i+1,j+1}
$
not larger than the $\eps_0$ returned by Lemma~\ref{lem:OSRIL} for both input $\nua_{i+1,j}$ and $d'$, and for input $\nua_{i,j+1}$ and $d'$, and not larger than the $\eps_0$ returned by Lemma~\ref{lem:TSRIL} for input $\nua_{i+1,j+1}$ and $d'$.
Choose $\nu_0=\min\big(\nua_{0,0},\tfrac{d'}{2},10^{-5}\gamma\big)$.
Now, Lemma~\ref{lem:SRLb} with input $\nu_0^2/(16\ell^2)$ and $2\ell$ returns $t_1$.

Set $\zeta= \big(\tfrac{d'}{4}\big)^\Delta /4 t_1$.
Given $\eps'$, let $\epsaa_{\Delta} = \eps'$ and for every $i \in (\Delta-1, \dots,1,0)$, let $\epsaa_{i}\le\epsaa_{i+1}$ be returned by Lemma~\ref{lem:OSRIL} with input $\eo= \epsaa_{i+1}$ and $\ao = d$.
Next, let $\epsa_{\Delta-1,\Delta-1}=\eps'$ and $\epsa_{i,\Delta}=\epsa_{\Delta,i}=1$ for $i \in [\Delta]$.
For each $(i,j) \in\{0, \dots, \Delta-1\}^2\setminus\{(\Delta-1,\Delta-1)\}$ in reverse lexicographic order, we choose~$\epsa_{i,j}\le\epsa_{i+1,j},\epsa_{i,j+1},\epsa_{i+1,j+1}
$
not larger than the $\eps_0$ returned by Lemma~\ref{lem:OSRIL} for both input $\epsa_{i+1,j}$ and $d$, and for input $\epsa_{i,j+1}$ and $d$, and not larger than the $\eps_0$ returned by Lemma~\ref{lem:TSRIL} for input $\epsa_{i+1,j+1}$ and $d$.

We choose $\eps_0 \leq \epsaa_0, \epsa_{0,0}, \frac{\nu_0}{2t_1}$ small enough such that $(1+\eps_0)^{\Delta} \leq 1+\eps'$ and $(1-\eps_0)^{\Delta} \geq 1-\eps'$.
Given $r\ge 10^5\gamma^{-1}$, $\eps$ with $0<\eps\le\eps_0$, and $\mu > 0$, let $C$ be a large enough constant for all of the above calls to Lemmas~\ref{lem:OSRIL} and~\ref{lem:TSRIL}, and for Proposition~\ref{prop:chernoff} with input $\eps_0$.
Finally, we choose $\Ca=10^{10\Delta}d'^{-\Delta}\ell t_1r\mu^{-1}$.

Let $\Gamma = G(n,p)$ with $p \geq \Ca {(\log n/n)}^{1/\Delta}$.
Then $\Gamma$ satisfies a.a.s.~the properties stated in Lemma~\ref{lem:OSRIL}, Lemma~\ref{lem:TSRIL}, Proposition~\ref{prop:chernoff} and Lemma~\ref{lem:SRLb} with the parameters specified above.
We assume from now on that $\Gamma$ satisfies these good events and has these properties.
Let $G'$, $v\in V(G')$, $G$, $\{V_i\}_{i\in \{0, \ldots, r\}}$, $H'$, $x \in V(H')$, the $k$-colouring $\rho$ of $V(H')$, and the $(\eps,d,p)$-reduced graph $R$, be as in the statement of the lemma. Since $\eps\le\eps_0$, $R$ is also an $(\eps_0,d,p)$-reduced graph.

To be able to apply Lemma~\ref{lem:SRLb} we need to choose a suitable subset of the clusters $\{V_i\}_{i\in \{0, \ldots, r\}}$ of bounded size.
As the clusters $\{V_i\}_{i\in \{0, \ldots, r\}}$ might be of different sizes and we will want to have a minimum degree condition on the reduced graph, we will consider a weighted version of this degree that takes the cluster sizes into account.
\begin{claim}\label{claim:Vast}
    There exists $L \subset [r]$ of size $\ell$ such that
    $R^\ast:=R[L]$ satisfies the following weighted minimum degree condition, where $V^\ast = \bigcup_{i \in L} V_i$.
    \[
        \forall i \in L:
        \sum_{j \in N_R(i) \cap L} \frac{|V_j|}{|V^\ast|}
        \ge
        \left(\frac{k-1}{k} + \frac{\gamma}{5}\right)\,.
    \]
    Additionally, we have that
    \[
        W:=
        \left\{
            w \in N_{G'}(v): |N_{G'}(w) \cap V^\ast| \ge \left(\frac{k-1}{k} + \frac{\gamma}{5}\right)p|V^\ast \cap V(G')|
        \right\}
    \]
    has size at least $(1 - \xi) |N_{G'}(v)|$    and there are at least $\frac12 \gamma p^{\binom{s+1}{2}}{(\mu n)}^s$ copies of $K_s$ in $W$.
\end{claim}
\begin{claimproof}
We choose a subset $L \subset [r]$ of size $\ell$ uniformly at random.
First, we will transfer the minimum degree of $G$ to the reduced graph and show that with high probability the minimum degree is preserved on the chosen clusters.
Recall that $G$ satisfies a minimum degree of $\delta(G) \ge (\frac{k-1}{k} + \gamma)pn$ and that we have the following bounds on the sizes of the clusters.
\begin{equation}\label{eq:cluster-sizes}
    \frac{4n}{r}
    \ge
    |V_i|
    \ge
    \frac{n}{4r}
    \ge
    Cp^{-1} \log n
\end{equation}
Without loss of generality, we may assume that no $V_i$ forms an irregular pair with more than $\sqrt{\eps}$ of the clusters, otherwise, add it to $V_0$, which over all clusters increases the size of $V_0$ by at most $4\sqrt{\eps}n$.
Fix $i \in [r]$.
Proposition~\ref{prop:chernoff} applied to the edges between $V_i$ and $V_0$ implies that
\[
    e(V_i, V_0)
    \le
    2p (\eps + 4\sqrt{\eps})n |V_i|
    \quad
    \text{and}
    \quad
    e(V_i)
    \le
    2p {|V_i|}^2
    \le 
    2p \frac{16}{r}n |V_i|
\]
Also, we can bound the number of edges from $V_i$ to other clusters that are in pairs which are not dense or $(\eps,p)$-regular as follows.
\[
    e\Big(V_i, \bigcup_{j \in R \setminus N_R(i)} V_j\Big)
    \le
    dpn |V_i| + 2p \cdot 4\sqrt{\eps} n |V_i|.
\]
Putting the above together, we obtain that 
\[
    e\Big(V_i, \bigcup_{j \in N_R(i)} V_j\Big)
    \ge
    \left(\frac{k-1}{k} + \gamma - 2\eps - 16\sqrt{\eps} - d - \frac{32}{r}\right)pn |V_i|
\]
As, again by Proposition~\ref{prop:chernoff}, the number of edges between any $V_i$ and $V_j$ is at most $(1+\eps_0)|V_i||V_j|$, we get that
\[
    \sum_{j \in N_R(i)} \frac{|V_j|r}{|V(G)|}
    \ge
    \left(\frac{k-1}{k} + \gamma - 2\eps - 16\sqrt{\eps} - d - \frac{32}{r}\right)
    {(1+\eps_0)}^{-1} r 
    \ge
    \left(\frac{k-1}{k} + \frac{\gamma}{2} \right) r.
\]
By the size conditions on the clusters, the relative sizes $w_j := \frac{|V_j|r}{|V(G)|}$ take values in $(\frac 14,4)$.
We now consider
\[
    w'_j = \xi \left\lfloor {w_j}/{\xi} \right\rfloor,
\]
the discretisation of $w_j$ into steps of size~$\xi$.
Of these discretised weights, we will ignore those that occur fewer than $\xi^2 r$ times.
We lose at most a factor of $4\xi$ due to the discretisation as all weights are at least $\frac 14$.
Also weights in $(\frac 14,4)$ occuring fewer than $\xi^2 r$ times contribute at most $16 \xi r$ to the sum, so we get the following lower bound.
\[
    \sum_{j \in N_R(i)} w'_j
    \ge
    \left(1-4\xi\right)
    \left(\frac{k-1}{k} + \frac{\gamma}{2} \right) r - 16 \xi r
    \ge
    \left(\frac{k-1}{k} + \frac{\gamma}{3} \right) r.
\]
We can now apply the hypergeometric inequality (Theorem~\ref{thm:hypergeometric}) to all possible rounded weight values separately.
For any $j \in [r]$ the probability that $j$ is in $L$ is $\ell/r$ and so for a given density in $(\frac 14, 4)$, which occurs, say, $\theta r$ times, the probability that this density is chosen fewer than $(1-\xi)\theta\ell$ times is at most $2e^{-\xi^2 \cdot \xi \theta \ell/3} \le 2e^{-\xi^5 \ell/3}$.
This implies by the union bound that with probability at most $4\xi^{-1} 2e^{-\xi^5 \ell/3}$ we do not have
\begin{equation}\label{eq:weighted-min-degree}
    \sum_{j \in N_R(i) \cap L} w_j
    \ge
    (1-\xi)\left(\frac{k-1}{k} + \frac{\gamma}{3} \right) \frac{\ell}{r} r
    \ge
    \left(\frac{k-1}{k} + \frac{\gamma}{4} \right) \ell.
\end{equation}
So by the union bound the expected number of vertices in $R^\ast$ that do not satisfy~\eqref{eq:weighted-min-degree} is at most $\ell 8 \xi^{-1} e^{-{\xi^5 \ell}/{3}}<1/10$, where the inequality is by choice of $\ell$.
By Markov's inequality, the probability that there is any such vertex in $R^\ast$ is thus at most~$1/10$. 
By the same discretisation of $w_j$ and application of the hypergeometric inequality to the discretised weights, we can also deduce that
\begin{equation}\label{eq:size-Vast}
    |V^\ast|
    =
    \frac{|V(G)|}{r}
    \sum_{i \in L} w_i
    =
    (1 \pm 100\xi)
    \frac{\ell}{r}
    \sum_{i \in [r]} w_i
    =
    (1 \pm 100\xi)
    (1 \pm \eps)
    \frac{\ell |V(G)|}{r}
\end{equation}
with probability at least $9/10$.
Putting~\eqref{eq:weighted-min-degree} and~\eqref{eq:size-Vast} together implies that with probability at least $8/10$ the first claimed statement holds.

For the claim, we also require that the minimum degree condition of the vertices in $N_{G'}(v)$ carries over to the chosen clusters for most vertices.
Fix $w$ in $N_{G'}$.
For $j \in [r]$ we consider the following weighted $p$-density, which may take values in $(0,5)$.
\[
    {d}_{w,j} = {d}_{G,p}(\{w\},V_j \cap V(G')) \frac{|V_j \cap V(G')|r}{|V(G')|}.
\]
Accounting for the exceptional set $V_0$ with Proposition~\ref{prop:chernoff}, the minimum degree condition on $G'$ of $(\frac{k-1}{k} + {\gamma})p|V(G')|$ implies that these weighted $p$-densities satisfy
\[
    \sum_{j \in [r]}  {d}_{w,j}
    \ge
    \left(\frac{k-1}{k} + {\gamma} - 2\eps\right)r
    \ge
    \left(\frac{k-1}{k} + \frac{\gamma}{2}\right)r.
\]
Similarly to before, we consider ${d'}_{w,i} = \xi \lfloor {{d}_{w,i}}/{\xi} \rfloor$, the discretisation of ${d}_{w,i}$ into steps of size~$\xi$.
Of these discretised weighted densities, we ignore those that occur fewer than $\xi^2 r$ times and those that are smaller than $\sqrt{\xi}$.
The small densities contribute at most $\sqrt{\xi}r$ to the sum and we lose a factor of at most $\sqrt{\xi}$ due to the discretisation for larger values.
Also weights in $(\sqrt{\xi},5)$ occuring fewer than $\xi^2 r$ times contribute at most $25\xi r$ to the sum, so we get the following lower bound.
\[
    \sum_{i \in [r]} {d'}_{w,i}
    \ge
    (1 - \sqrt{\xi}) \left(\frac{k-1}{k} + \frac{\gamma}{2} - \sqrt{\xi} - 25 \xi\right)r
    \ge
    \left(\frac{k-1}{k} + \frac{\gamma}{3}\right)r.
\]
Applying the hypergeometric inequality to all density values separately as before, we get that for any $w \in N_{G'}(v)$ with probability at most $5\xi^{-1} 2e^{-\xi^5 \ell/3} \ge \xi / 10$ we do not have
\begin{equation}\label{eq:W-weighted-minimum-degree}
    \sum_{i \in L} {d'}_{w,i}
    \ge
    \left( 1 - \xi\right) \left(\frac{k-1}{k} + \frac{\gamma}{3}\right) \frac{\ell}{r} r
    \ge
    \left(\frac{k-1}{k} + \frac{\gamma}{4}\right) \frac{\ell}{r} r.
\end{equation}
So the expected number of vertices in $N_{G'}(v)$ not satisfying~\eqref{eq:W-weighted-minimum-degree} is at most $\xi |N_{G'}(v)| / 10$.
By Markov's inequality, with probability at least $9/10$ at most a fraction $\xi$ of vertices in $N_{G'}(v)$ violate~\eqref{eq:W-weighted-minimum-degree}.
And in particular all vertices satisfying~\eqref{eq:W-weighted-minimum-degree} have at least
\[
    (1 - 100\xi)
    (1 - \eps)
    \left(\frac{k-1}{k} + \frac{\gamma}{4}\right)
    (1 - \eps)
    \mu p|V^\ast|
    \ge
    \left(\frac{k-1}{k} + \frac{\gamma}{5}\right)p|V^\ast \cap V(G')|
\]
neighbours in $V^\ast \cap V(G')$ if~\eqref{eq:size-Vast} holds.
So indeed with probability at least $7/10$ the first two claimed statements hold, so assume we chose $L$ such that they do.

For the claim it only remains to show the lower bound on the number of cliques in $W$.
It follows, by inductively building up cliques, from the assumption in the lemma that any $t \le \Delta$ vertices of $G'$ have at most $2 p^t \mu n$ common neighbours in $G'$, that $v$ and each $w \in N_{G'}(v)$ are contained in at most
\[
    \prod_{t=2}^{s}
    2 p^t \mu n
    =
    p^{\binom{s+1}{2}-1} {(2\mu n)}^{s-1}
\]
copies of $K_{s+1}$.
Since $|W|\ge (1-\xi)|N_{G'}(v)|$, the number of copies of $K_s$ which are in $N_{G'}(v)$ but not $W$ is at most $\xi|N_{G'}(v)|\cdot p^{\binom{s+1}{2}-1}(2\mu n)^{s-1}$. Since $|N_{G'}(v)|\le 2\mu pn $, and $N_{G'}(v)$ contains at least $\gamma p^{\binom{s+1}{2}}(\mu n)^s$ copies of $K_s$,  there are at least 
\[
    \gamma p^{\binom{s+1}{2}}{(\mu n)}^s
    - \xi \cdot2\mu pn\cdot p^{\binom{s+1}{2}-1} {(2\mu n)}^{s-1}
    \ge
    \tfrac12 \gamma p^{\binom{s+1}{2}} {(\mu n)}^s
\]
copies of $K_s$ in $W$.
\end{claimproof}
Let $\{W_i\}_{i \in [\ell]}$ be an arbitrary equipartition of $W$ into $\ell$ parts (so that the fine partition we are about to obtain has enough parts in $W$).
We apply Lemma~\ref{lem:SRLb} to $G'$ with the $2\ell$-part initial partition $\{ (V_i \cap V(G')) \setminus W \}_{i\in L} \cup \{W_i\}_{i \in [\ell]}$ and input parameter $\nu_0^2/(16\ell^2)$.
This returns a partition refining each of these sets into $1\le t\le t_1$ clusters $\{V_{i,j}\}_{i\in L,j\in[t]}\cup\{W_{i,j}\}_{i\in[\ell],j\in[t]}$ together with small exceptional sets~$\{V_{i,0}:i\in L\}\cup\{ W_{i,0}: i\in[\ell]\}$.
From the definition of a regular refinement, there are at most $\tfrac{\nu_0^2}{16\ell^2}\cdot(2\ell t)^2$ irregular pairs in this partition, and in particular at most $\nu_0 t$ of the clusters form an irregular pair with more than $\nu_0 t$ of the clusters.
Include the vertices of all those clusters in the exceptional sets, which now make up a fraction of at most $2 \nu_0$ of the vertices.

We now want to obtain $s$ clusters $W'_1,\dots,W'_s$ in ${\{W_{i,j}\}}_{i \in [\ell],j \in [t]}$ that are pairwise $(\nu_0, d', p)$-regular.
Assume for a contradiction that no such clusters exist.
So each $K_k$ in $W$ must either contain an edge meeting an exceptional set $W_{i,0}$, one which does not lie in a $(\nu_0, d', p)$-regular pair or one that is contained completely in some set $W_{i,j}$ for $i\in[\ell]$ and $j\in[t]$.
Note that we have for all $i \in [\ell]$ and $j \in [t]$ that
\[
    |W_{i,j}| \ge \frac{1}{2 \ell t_1} |W| \ge \frac{\mu np}{4\ell t_1} \ge Cp^{-1} \log n.
\]
So we may apply Proposition~\ref{prop:chernoff} to bound the number of edges within and between clusters.
Using the upper bound on common neighbourhoods in $G'$ given in the lemma statement to bound the number of edges meeting the exceptional sets, we obtain that deleting at most
\[
    2 \nu_0 |W| 2p^2 \mu n
    + 2p(\nu_0 + d') {|W|}^2
    + \ell 2p {(|W|/\ell)}^2
    \le
    (8\nu_0 + 8\nu_0 + 8d' + {2}/{\ell}) p^3 \mu^2 n^2
\]
edges would remove all cliques from $W$.
Again by the upper bound on common neighbourhoods in $G'$ given in the lemma any of these edges is contained in at most
\[
    \prod_{t=3}^{s}
    2 p^t \mu n
    =
    p^{\binom{s+1}{2}-3} {(2\mu n)}^{s-2}
\]
copies of $K_{s+1}$ together with $v$.
So there would be at most
\[
    (16\nu_0 + 8d' + {2}/{\ell}) p^3 \mu^2 n^2 p^{\binom{s+1}{2}-3} {(2\mu n)}^{s-2}
    <
    \tfrac12 \gamma p^{\binom{s+1}{2}} {(\mu n)}^s
\]
copies of $K_s$ in $W$, a contradiction. It follows that there are some $s$ clusters in $W$ which are pairwise $(\nu_0,d',p)$-regular. Let $W'_1,\dots,W'_s$ in ${\{W_{i,j}\}}_{i \in [\ell],j \in [t]}$ be pairwise $(\nu_0, d', p)$-regular.

Because the vertices of $W$ each have at least $\big(\tfrac{k-1}{k}+\tfrac{\gamma}{5}\big)p|V^*\cap V(G')|$ $G'$-neighbours in $V^*$, the number of edges leaving each cluster $W'_i$ to $V^*$ is at least $|W'_i|\big(\tfrac{k-1}{k}+\tfrac{\gamma}{5}\big)p|V^*\cap V(G')|$. By Proposition~\ref{prop:chernoff}, and because at most $\nu_0t$ irregular pairs leave $W'_i$, at most $(1+\eps_0)p|W'_i|\nu_0|V^*\cap V(G')|$ of these edges lie in irregular pairs. By definition, at most $d'p|W'_i||V^*\cap V(G')|$ of these edges lie in pairs of relative density less than $d'$. Thus the remaining edges lie in $(\nu_0,d',p)$-regular pairs, and there are at least $|W'_i|\big(\tfrac{k-1}{k}+\tfrac{\gamma}{6}\big)p|V^*\cap V(G')|$ of these edges. Since the number of edges between $W'_i$ and any given $V_{i',j'}$ is at most $(1+\eps_0)p|W'_i||V_{i',j'}|$ by Proposition~\ref{prop:chernoff}, we obtain
\begin{equation}\label{eq:weighted-regular-min-degree}
    \sum_{V_{i',j'}:\, (W'_i, V_{i',j'})\text{ is }(\nu_0, d', p)-\text{regular}}
    \frac{|V_{i',j'}|}{|V^\ast \cap V(G')|}
    \ge
    \left(\frac{k-1}{k} + \frac{\gamma}{8}\right)\,.
\end{equation}

Now we can choose the clusters into which we will embed the vertices of $H'$. We choose sequentially
\[
    (\kq_{s+1},j_{s+1}),\dots,(\kq_{k+1},j_{k+1}) \in L \times [t]
\]
such that for each $1\le i\le s$ and each $s+1\le i'\le k+1$, the pair $\big(W'_i,V_{\kq_{i'},j_{i'}}\cap V(G')\big)$ is $(\nu_0,d',p)$-regular, and for each $s+1\le i'<i''\le k+1$ the pair $(\kq_{i}',\kq_{i''})$ is an edge of $R^*$. This is possible by~\eqref{eq:weighted-regular-min-degree} and Claim~\ref{claim:Vast}, which give a weighted minimum degree condition that implies that for any $k$ clusters (in $W$ or $V^*$ or a mixture) there is a cluster in $V^*$ which satisfies the given condition with respect to all $k$ clusters.

We then choose pairs $(\kq_{s},j_{s}),\dots,(\kq_{1},j_{1})$ in that order sequentially such that for each $a \in \{s, \dots, 1\}$ the clusters
\[
    W'_1,\dots,W'_{a-1}, V_{\kq_{a},j_{a}},V_{\kq_{a+1},j_{a+1}},\dots,V_{\kq_{k+1},j_{k+1}}
\]
satisfy the same condition, i.e.\ for each $1\le i\le a$ and each $a+1\le i'\le k+1$, the pair $\big(W'_i,V_{\kq_{i'},j_{i'}}\cap V(G')\big)$ is $(\nu_0,d',p)$-regular, and for each $a+1\le i'<i''\le k+1$ the pair $(\kq_{i}',\kq_{i''})$ is an edge of $R^*$. Note that by choice of $\eps_0$, if $(\kq_{i'},\kq_{i''})$ is an edge of $R^*$ then the pair $\big(V_{\kq_{i'},j_{i'}}\cap V(G')\setminus W,V_{\kq_{i''},j_{i''}}\cap V(G')\setminus W\big)$ is $(\nu_0,d',p)$-regular in $G'$. For convenience, we let $V'_i:=V_{\kq_i,j_i}\cap V(G')\setminus W$ for each $1\le i\le k+1$.

We will embed $H'-(\{x\}\cup T)$ into the chosen clusters, i.e.\ $W'_1,\dots,W'_s,V'_1,\dots,V'_{k+1}$, using the regularity embedding strategy mentioned above. We will need to embed some vertices of $H'$ which are not neighbours of $x$ into the sets $W'_i$. For this to work, each such vertex $u$ needs to have at most $\Delta-3$ neighbours which we embed before $u$, and the aim of the next arguments is to assign vertices of $H'$ to clusters, and put an order on $V(H')$, which ensures this.

Recall that $\rho$ is a proper $k$-colouring of $H'$ which uses only $s$ colours on $N(x)$. Reordering the colours if necessary, let us assume $\rho$ uses only colours in $[s]$ on $N(x)$.
We define a proper $(k+1)$-vertex colouring $\rho':V(H') \to [k+1]$ inductively as follows.
Initially we set $\rho'(w)=\rho(w)$ for all $w$ in $H'$.
Let
\[
    U_{\rho'} = \bigcup_{i=2}^{s} \left\{ w \in N^i(x): \rho'(w) \le s-i+1 \right\},
\]
where $N^i(x)$ refers to the vertices at distance $i$ from $x$.
If $U_{\rho'}$ contains a vertex $w$ with no neighbour in ${\rho'}^{-1}(i)$ for some $\rho'(w)+1\le i\le k+1$, we set $\rho'(w)=i$ (if there are several such $i$, we choose one arbitrarily). We repeat this step until $U_{\rho'}$ contains no such vertices. Since the colour of any given vertex only increases through this process, the recolouring procedure must terminate eventually. The resulting $\rho'$ has the following property: if $u$ is any vertex with $d(x,u)\ge 2$ and $d(x,u)+\rho'(u)\le s+1$, then $u$ has a neighbour in each of the colour classes $\rho'(u)+1,\dots,k+1$. In particular, since $\rho'(u)\le s-1$ (as otherwise $d(x,u)+\rho'(u)\le s+1$ is impossible), and since $s\le k-1$ by assumption of the lemma, $u$ has a neighbour in each of the colour classes $k-1$, $k$ and $k+1$. Observe that no vertex in these colour classes is in $U_{\rho'}$ by definition.
 
Note that the colouring remains unchanged on $N(x)$ and the vertices at distance $s+1$ from $x$.
We define an order $<_{\rho'}$ on $V(H')\setminus \{x\}$ by putting first all the vertices of $U_{\rho'}$ in an arbitrary order, then the remaining vertices of $V(H')\setminus (T\cup\{x\})$ in an arbitrary order, and finally the vertices of $T$ in an arbitrary order. With the colouring $\rho'$ defined as above, this gives us, for all $u$ at distance at least two from $x$ with $\rho'(u) + d(x,u) \le s+1$:
\begin{equation}
    \label{eq:order}
    |\text{pred}_{<_{\rho'}}(u) \cap N(u)| = |\{u': u' <_{\rho'} u, u' \in N(u)\}| \le \Delta -  3\,.
\end{equation}
Now we can assign the vertices of $H'$ to clusters.
For $u \in V(H')$, let 
\[
    V_u = V_{\kq_{\rho'(u)}}
    \quad \text{and} \quad
    C_u =
    \begin{cases}
        W'_{\rho'(u)}
        & \text{if } \rho'(u) + d(x,u) \le s+1 \\
        V_{\kq_{\rho'(u)},j_{\rho'(u)}}
        & \text{otherwise}.\\
    \end{cases}
\]
We now iteratively embed the vertices of $H'$ in the order specified above respecting the assignments to clusters. The following claim, which we prove by induction on the number of embedded vertices, encapsulates the conditions we maintain through this embedding.
Here, as in the statement of the lemma, we set $\Pi(u)=\phi\big(N_{H'}(u)\cap \Dom(\phi)\big)$, and recall that $T$ is the vertices in $H'$ at distance exactly $s+1$ from $v$.

\begin{claim}\label{claim:coverv}
    For each integer $0 \le z \le |V(H')\setminus T|-1$ there exists an embedding $\phi$ of the first $z$ vertices of~$H'\setminus (T\cup\{x\})$ (w.r.t.\ to the order $<_{\rho'}$) into $G$ such that
    \begin{enumerate}[label=\itmarab{I}]
        \item\label{item:claim-coverv-right-cluster} for every $u \in \Dom(\phi)$ we have $\phi(u) \in C_{u}$,
    \end{enumerate}
    and for every $u,u' \in H'\setminus (\Dom(\phi)\cup\{x\})$, where $u' \in N_{H'}(u)$ we have the following.
    \begin{enumerate}[label=\itmarab{I},start=2]
        \item\label{item:claim-coverv-2} $|N_G(\Pi(u),C_u)| \geq \left(\frac{d'}4\right)^{|\Pi(u)|} p^{|\Pi(u)|}|C_u|$,
        \item\label{item:claim-coverv-3} $|\NGa(\Pi(u),C_u)| = {(1\pm\nu_0)}^{|\Pi(u)|}p^{|\Pi(u)|} |C_u|$,
        \item\label{item:claim-coverv-4} $\big(\NGa(\Pi(u),C_u),\NGa(\Pi(u'),C_{u'})\big)$ is $(\nua_{|\Pi(u)|,|\Pi(u')|},d',p)_G$-regular.
    \end{enumerate}
    Also, if $d(x,u)+\rho'(u),d(x,u')+\rho'(u')> s+1$ we have
    \begin{enumerate}[label=\itmarab{L}]
        \item\label{item:claim-coverv-1prime} if $|\Pi(u)|\le\Delta-1$ then $\big(\NGa(\Pi(u),V_u),V_{\kq_j}\big)$ is $(\epsaa_{|\Pi(u)|}, d, p)_G$-regular for each $j\neq\rho'(u)$,
        \item\label{item:claim-coverv-3prime} $|\NGa(\Pi(u),V_u)| = {(1\pm\eps_0)}^{|\Pi(u)|}p^{|\Pi(u)|} |V_u|$,
        \item\label{item:claim-coverv-4prime} $\big(\NGa(\Pi(u),V_u),\NGa(\Pi(u'),V_{u'})\big)$ is $(\epsa_{|\Pi(u)|,|\Pi(u')|},d,p)_G$-regular.
    \end{enumerate}
\end{claim}
\begin{claimproof}
We prove the claim inductively, starting with $z=0$ and $\phi$ the empty embedding. We first check that the claimed properties hold for this embedding. \ref{item:claim-coverv-right-cluster} is true vacuously. Since $\Pi(u)=\emptyset$ for each $u\in V(H')\setminus\{x\}$, the various neighbourhoods in $C_u$ and $C_{u'}$ are equal to $C_u$ and $C_{u'}$. So~\ref{item:claim-coverv-2} and~\ref{item:claim-coverv-3} hold trivially, and~\ref{item:claim-coverv-4} holds by choice of the $W'_i$ and by choice of $\nua_{0,0}$. Similarly,~\ref{item:claim-coverv-1prime} and~\ref{item:claim-coverv-4prime} hold because by choice of the $\kq_j$ the pair  $(V_{\kq_j},V_{\kq_{j'}})$ is $(\eps,d',p)$-regular for each $1\le j<j'\le k+1$, and~\ref{item:claim-coverv-3prime} holds trivially.

We now have to show the induction step holds; suppose that for some $0\le z<|V(H')\setminus T|-1$, the map $\phi$ is an embedding of the first $z$ vertices of $H'-(T\cup \{x\})$ satisfying the conclusion of Claim~\ref{claim:coverv}. Let $w$ be the $(z+1)$st vertex of $H'-(T\cup\{x\})$. We aim to show the existence of an embedding $\phi'$ extending $\phi$ satisfying the conclusion of Claim~\ref{claim:coverv} for $z+1$.

To do this, it is enough to show that, for each statement among~\ref{item:claim-coverv-2}--\ref{item:claim-coverv-4} and~\ref{item:claim-coverv-1prime}--\ref{item:claim-coverv-4prime} separately, the number of vertices in $N_G\big(\Pi(w),C_w\big)$ which cause the given statement to fail is small compared to $\big|N_G\big(\Pi(w),C_w\big)\big|$; then we choose a vertex $y$ in that set (so guaranteeing~\ref{item:claim-coverv-right-cluster}) which causes none of the statements to fail, and have the desired embedding $\phi\cup\{w\to y\}$. We therefore record some lower bounds on $\big|N_G\big(\Pi(w),C_w\big)\big|$.

Suppose $d(x,w)\ge2$ and $d(x,w)+\rho'(w)\le s+1$, or if $d(x,w)=1$ and $w$ has two neighbours in $H'-x$ which come after $w$ in $<_{\rho'}$. In the first case,  by~\eqref{eq:order}, we have $|\Pi(w)|\le\Delta-3$. In the second case, since $w$ has three neighbours in $H'$ which do not come before it in $<_{\rho'}$ (as $x$ is not in that order at all) we have $|\Pi(w)|\le\Delta-3$. In either case, by~\ref{item:claim-coverv-2}, we get
\begin{equation}\label{eq:coverv:two}
\big|N_G\big(\Pi(w),C_w\big)\big|\ge\big(\tfrac{d'}{4}\big)^{\Delta-3}p^{\Delta-3}|C_w|\ge\big(\tfrac{d'}{4}\big)^{\Delta-3}p^{\Delta-2}\cdot\tfrac{\mu n}{4\ell t_1}\ge 100C\Delta^2p^{-2}\log n\,,
\end{equation}
where the final inequality uses $p\ge C^*\big(\tfrac{\log n}{n}\big)^{1/\Delta}$ and the choice of $C^*$.
By a similar calculation, if either $d(x,w)=1$ and $w$ has a neighbour coming after in in $<_{\rho'}$, or $d(x,w)+\rho'(w)>s+1$ and $w$ has a neighbour coming after it in $<_{\rho'}$, we have
\begin{equation}\label{eq:coverv:one}
\big|N_G\big(\Pi(w),C_w\big)\big|\ge\big(\tfrac{d'}{4}\big)^{\Delta-1}p^{\Delta-1}\cdot\tfrac{\mu n}{4\ell t_1r}\ge 100C\Delta^2p^{-1}\log n\,.
\end{equation}
Finally, if either $d(x,w)=1$ or $d(x,w)+\rho'(w)>s+1$, we get
\begin{equation}\label{eq:coverv:none}
\big|N_G\big(\Pi(w),C_w\big)\big|\ge\big(\tfrac{d'}{4}\big)^{\Delta}p^{\Delta}\cdot\tfrac{\mu n}{4\ell t_1r}\ge 100C\Delta^2\log n\,.
\end{equation}

We now estimate the fraction of $\big|N_G\big(\Pi(w),C_w\big)\big|$ which causes each of the desired statements to fail. The statement~\ref{item:claim-coverv-2} can only fail for a neighbour $u$ of $w$, and then only if we choose $y\in N_G\big(\Pi(w),C_w\big)$ which has too few neighbours in $N_G\big(\Pi(u),C_u\big)$. But by~\ref{item:claim-coverv-4} these two sets are on either side of a $\big(\nu^*_{|\Pi(w)|,|\Pi(u)|},d',p)_G$-regular pair, and by~\ref{item:claim-coverv-2} and~\ref{item:claim-coverv-3} the latter covers more than a $\nu^*_{\Delta,\Delta}$-fraction of $\NGa\big(\Pi(u),C_u\big)$. So by regularity, at most $\nu^*_{\Delta,\Delta}|\NGa\big(\Pi(w),C_w\big)\big|$ vertices of $\big|N_G\big(\Pi(w),C_w\big)\big|$ can cause~\ref{item:claim-coverv-2} to fail for $u$. Using~\ref{item:claim-coverv-2} and~\ref{item:claim-coverv-3}, and summing over the at most $\Delta$ choices of $u$, we see that at most a $8^\Delta d'^{-\Delta}\Delta\nu^*_{\Delta,\Delta}$-fraction of $\big|N_G\big(\Pi(w),C_w\big)\big|$ cause~\ref{item:claim-coverv-2} to fail.

For~\ref{item:claim-coverv-3}, we note that embedding $w$ can only cause this statement to fail if $w$ has at least one neighbour in $H'$ coming after it in $<_{\rho'}$, and in this case by~\eqref{eq:coverv:two} and~\eqref{eq:coverv:one}, we have $\big|N_G\big(\Pi(w),C_w\big)\big|\ge 100C\Delta^2p^{-1}\log n$. Now a vertex $y\in N_G\big(\Pi(w),C_w\big)$ can only cause~\ref{item:claim-coverv-3} to fail if it has the wrong number of neighbours in $\NGa\big(\Pi(u),C_u\big)$ for some neighbour $u$ of $w$. Because the good event of Proposition~\ref{prop:chernoff} occurs, this happens for at most $Cp^{-1}\log n$ vertices, and summing over the at most $\Delta$ choices of $u$, we see that at most a $\tfrac{1}{100}$-fraction of  $\big|N_G\big(\Pi(w),C_w\big)\big|$ cause~\ref{item:claim-coverv-3} to fail.

For~\ref{item:claim-coverv-4}, we need to be a bit more careful. To start with, if there are no neighbours of $w$ coming after $w$ in $<_{\rho'}$, then no matter how we embed $w$ we cannot make \ref{item:claim-coverv-4} fail. Suppose first that there are neighbours of $w$ coming after $w$ in $<_{\rho'}$, but that no two such neighbours are adjacent. As above, by~\eqref{eq:coverv:two} and~\eqref{eq:coverv:one}, we have $\big|N_G\big(\Pi(w),C_w\big)\big|\ge 100C\Delta^2p^{-1}\log n$. By~\ref{item:claim-coverv-4}, a vertex $y\in N_G\big(\Pi(w),C_w\big)$ can only cause~\ref{item:claim-coverv-4} to fail for a given $u,u'$ if $u$ is a neighbour of $w$ and $u'$ is not, and $y$ is one of the at most $Cp^{-1}\log n$ vertices which fail to inherit regularity, as guaranteed by the good event of Lemma~\ref{lem:OSRIL}. Summing over the at most $\Delta^2$ choices of $u,u'$, we see that in this case at most a $\tfrac{1}{100}$-fraction of  $\big|N_G\big(\Pi(w),C_w\big)\big|$ cause~\ref{item:claim-coverv-4} to fail. The remaining case is that there are two adjacent neighbours of $w$ coming after $w$ in $<_{\rho'}$. In this case we need the good events of Lemmas~\ref{lem:OSRIL} and~\ref{lem:TSRIL}, and consequently for given $u,u'$ up to $Cp^{-2}\log n$ vertices might fail to inherit regularity. But in this case by~\eqref{eq:coverv:two} we have$ \big|N_G\big(\Pi(w),C_w\big)\big|\ge 100C\Delta^2p^{-2}\log n$, and again in this case at most a $\tfrac{1}{100}$-fraction of  $\big|N_G\big(\Pi(w),C_w\big)\big|$ cause~\ref{item:claim-coverv-4} to fail.

The proofs that at most a $\tfrac{1}{100}$-fraction of $\big|N_G\big(\Pi(w),C_w\big)\big|$ cause any one of~\ref{item:claim-coverv-1prime}--\ref{item:claim-coverv-4prime} are essentially identical, and we omit the details.

Summing up, by choice of $\nu^*_{\Delta,\Delta}$ and since $|V(H')|\le\sum_{i=0}^{s+1}\Delta^i$, we see that at least half of $\big|N_G\big(\Pi(w),C_w\big)\big|$ consists of vertices $y$ such that $\phi\cup\{w\to y\}$ satisfies the conclusions of Claim~\ref{claim:coverv} for $z+1$, completing the induction step and hence the proof of the claim.
\end{claimproof}
Now we can conclude the proof of Lemma~\ref{lem:coverv}. Given an embedding of $H'-(T\cup\{x\})$ satisfying the conclusions of Claim~\ref{claim:coverv}, we extend it to an embedding $\phi$ of $H'-T$ by setting $\phi(x) = v$. This is a valid embedding since we embedded all neighbours of $x$ to $W$, and we obtain~\ref{item:pel-v}.
Property~\ref{item:pel-clique} holds by choice of the $\kq_1,\dots,\kq_{k}$.
For every vertex $u$ in~$T$ we have that $C_u = V_{\kq_{\rho'(u)},j_{\rho'(u)}}$ and $|C_u| \ge |{V_{\kq_{\rho'(u)}}}\cap V(G')|/{2 t_1}$.
So by the choice of~$\zeta$,~\ref{item:pel-minsize} follows from~\ref{item:claim-coverv-2}.
The choice of constants ensures that the remaining statements in the lemma are a direct consequence of~\ref{item:claim-coverv-1prime}-\ref{item:claim-coverv-4prime}.
\end{proof}

\section{Proof of the main technical result}
\label{sec:mainproof}
The proof of Theorem~\ref{thm:maink} is broadly similar to the proof of~\cite[Theorem~23]{ABET}. Again, basically the idea is that we apply the lemmas of Section~\ref{sec:mainlemmas} in order to first find a well-behaved partition of $G$ and a corresponding partition of $H$. We then deal with the few badly-behaved vertices of $G$ by sequentially pre-embedding onto them some vertices of $H$ whose neighbourhoods contain at most $s$ colours. Lemma~\ref{lem:coverv} deals with this pre-embedding, and sets up for the vertices which are not pre-embedded but which have pre-embedded neighbours restriction sets in the sense of Definition~\ref{def:restrict}. We then adjust the partition of $H$ to fit this pre-embedding, and balance the partition of $G$ to match. Finally, we see that the conditions of Lemma~\ref{thm:blowup} are met, and that lemma completes the desired embedding of $H$ in $G$.

As in~\cite{ABET}, there are two slightly subtle points. The first is that for $\Delta=2$ we can have $Cp^{-2}>pn$, so that we should be worried that we come to some badly-behaved vertex of $G$ onto which we wish to pre-embed and discover that all its neighbours have already been used in pre-embedding. As in~\cite{ABET}, this is easy to handle: at each step we choose the badly-behaved vertex with most neighbours already embedded to. It is easy to check that this ordering avoids the above problem. The second, more serious, problem is that we need restriction sets fulfilling the conditions of Definition~\ref{def:restrict}. Although Lemma~\ref{lem:coverv} gives us pre-embeddings satisfying these conditions, we might destroy the conditions when we pre-embed later vertices. The condition we could destroy is simply that we need each restriction set to be reasonably large; the danger is that we pre-embed many vertices to some restriction set. The solution to this is (as in~\cite{ABET}) to select a set $S$, whose size is linear in $n$ but small, using Lemma~\ref{lem:hypgeo} to avoid large intersections with any possible restriction set. When we apply Lemma~\ref{lem:coverv} to cover a badly-behaved vertex $v$, we will pre-embed to $v$ and to some vertices chosen from $S$, and not to any other vertex. The badly-behaved vertices are not (by construction) in any restriction set, while $S$ has small intersection with all restriction sets, so that even removing all of $S$ would not make the restriction sets too small.

The only point in the proof where we really need to do more than in~\cite{ABET} (apart from using Lemma~\ref{lem:coverv} to pre-embed) is that we need to ensure the conditions of Lemma~\ref{lem:coverv} are met. When we wish to cover a badly-behaved $v$, its neighbourhood within the set $S$ must contain many copies of $K_s$. Further, some vertices of $S$ will have been used in earlier pre-embeddings, and we need to ensure that these used vertices do not hit too many of the copies of $K_s$. For this, we apply the sparse regularity lemma, Lemma~\ref{lem:SRLb}, to $G\big[N_G(v)\big]$ before choosing $S$. We will see that (since $N_G(v)$ contains many copies of $K_s$) we find a set of $s$ clusters in $N_G(v)$ such that all the pairs are relatively dense and regular. When we use Lemma~\ref{lem:hypgeo} to choose $S$, we also insist that $S$ contains a significant fraction of each of these clusters. The order in which we cover badly-behaved vertices ensures that a (slightly smaller but still) significant fraction of each cluster is not used by the previous pre-embedding; and we find the desired many copies of $K_s$ in $N_G(v)\cap S$ as a result.

As a final observation, Lemma~\ref{lem:coverv}~\ref{item:pel-minsize} gives us something which looks like an image restriction set suitable for Definition~\ref{def:restrict}---but it is a subset of $S$. A careful reader will see from the constant choices below that it is therefore too small for Lemma~\ref{thm:blowup}. However, the fact that $S$ is selected at random allows us to deduce the existence of a larger image restriction set which is suitable for Lemma~\ref{thm:blowup}.

\begin{proof}[Proof of Theorem~\ref{thm:maink}]
Given $\gamma>0$, we set $d^+=2^{-s-5}\gamma$ and
$\eps^+_{s-2}=16^{-s}(d^+)^{2s}/s$. For each $i=s-3,s-4,\ldots,0$
sequentially, let $0<\eps^+_i\le \eps^+_{i-1}$ be sufficiently small for
Lemma~\ref{lem:TSRIL} with input $d^+$ and $\eps^+_{i+1}$. Let
$\eps^+\le\eps^+_0$ be small enough for an application of
Lemma~\ref{lem:hypgeo} with input $d^+$ and $\eps^+_0$. Let $t_1^+$ be
returned by Lemma~\ref{lem:SRLb} for input $\eps^+$ and $\lceil
1/d^+\rceil$, and let $\alpha^+=\tfrac14d^+/t_1^+$. Let
$\gamma^+=2^{-4s^2}(d^+)^{-2s^2}(t_1^+)^{-s}$. Note we have $\gamma^+<\gamma$.

We now choose $d \le \frac{\gamma^+}{32}$ not larger than the $d$ given by Lemma~\ref{lem:G} for input~$\gamma$,~$k$ and~
$r_0:=10^5\gamma^{-1}$. We let $\alpha$ be the $\zeta$ returned by
Lemma~\ref{lem:coverv} for input $\Delta$, $k$, $s$, $\gamma^+$ and $d$. We
set $D=\Delta$ and let $\eBL$ be returned by Lemma~\ref{thm:blowup} for
input $\Delta$, $\Delta_{R'}=3k$, $\Delta_J=\Delta$,
$\vartheta=\tfrac{1}{100D}$, $\zeta=\tfrac{1}{4}\alpha$, $d$ and
$\kappa=64$. Next, putting $\eps^*:=\tfrac18\eBL$ into
Lemma~\ref{lem:coverv} (with earlier parameters as above) returns
$\eps_0>0$. We set $\eps=\min(\eps_0,d,\eps^*/4\Delta,1/100k)$, and set
$\eps^-\le\eps$ small enough for Lemma~\ref{lem:hypgeo} with input as above
and $d,\eps$. Now Lemma~\ref{lem:G}, for input $\eps^-$
and earlier constants as above, returns $r_1$. At last,
Lemma~\ref{lem:balancing}, for input $k$, $r_1$, $\Delta$, $\gamma$, $d$ and
$8\eps$, returns $\xi>0$. Without loss of generality, we may assume
$\xi<10(10kr_1)$, and set $\beta=10^{-12}\xi^2/(\Delta k^4 r_1^2)$. Let
$\mu=\eps^2/(100000kr_1)$. Next, suppose $\Ca$ is large enough for Lemma~\ref{lem:coverv}, and also to play the
r\^ole of $C$ in each of these other lemmas, and also for
Proposition~\ref{prop:chernoff} with input $\eps$, for
Lemma~\ref{lem:TSRIL} with input $d^+$ and each of $\eps^+_i$ for
$i=1,\ldots,s-2$, and for Lemma~\ref{lem:hypgeo} with input $\eps\mu^2$,
$\eps$, $\min(d,d^+)$ and $\Delta$.

We set $C=10^{100}k^2 r_1^2 \eps^{-2}\xi^{-1}\Delta^{1000k^3}\mu^{-\Delta}\Ca$ and $z=10/\xi$. Given $p\ge C\big(\tfrac{\log}{n}\big)$,
a.a.s.\ $\Gamma=G(n,p)$ satisfies the good events of each of the lemmas and
propositions listed above with each of the specified inputs.

In addition, for each set $W$ of at most $\Delta$ vertices of $G(n,p)$, the
size of the common neighbourhood $N_{G(n,p)}(W)$ is distributed as a
binomial random variable with mean $p^{|W|}(n-|W|)$. By
Theorem~\ref{thm:chernoff}, the probability that the outcome is
$(1\pm\eps)p^{|W|}n$ is at least $1-n^{-(\Delta+1)}$ for sufficiently large $n$. By the union bound, we
conclude that a.a.s.\ $G(n,p)$ satisfies
\begin{equation}\label{eq:nolargedegs}
 \text{for each $W\subset V\big(G(n,p)\big)$ with $|W|\le\Delta$ we have }\big|N_{G(n,p)}(W)\big|=(1\pm\eps)p^{|W|}n\,.
\end{equation}

Suppose that $\Gamma=G(n,p)$ satisfies these good events. Let $G$ be a spanning subgraph of $\Gamma$ such that $\delta(G)\ge\big(\tfrac{k-1}{k}+\gamma\big)pn$ and such that for each $v\in V(G)$ the neighbourhood $N_G(v)$ contains at least $\delta p^{\binom{s}{2}}(pn)^s$ copies of $K_s$. Let $H$ be a graph on $n$ vertices with $\Delta(H)\le\Delta$. Let $\sigma$ be a proper colouring of $V(H)$ using colours $\{0,\dots,k\}$, and let $\cL$ be a labelling of $V(H)$ with bandwidth at most $\beta n$ with the following properties. The colouring $\sigma$ is $(z,\beta)$-zero-free with respect to $\cL$, the first $\sqrt{\beta}n$ vertices of $\cL$ do not use the colour zero, and the first $\beta n$ vertices of $\cL$ contain $Cp^{-2}$ vertices whose neighbourhood contains only $s$ colours.

We now claim that for each $v\in V(G)$ we can find $s$ large subsets of $N_G(v)$ all pairs of which are dense and regular in $G$. This forms a `robust witness' that each vertex neighbourhood in $G$ contains many copies of $K_{s}$.
\begin{claim}
For each $v\in V(G)$, there exist sets $Q_{v,1},\dots,Q_{v,s}\subset N_G(v)$ each of size at least $\alpha^+pn$ such that for each $i<j$ the pair $(Q_{v,i},Q_{v,j})$ is $(\eps^+,d^+,p)$-regular in $G$.
\end{claim}
\begin{claimproof}
 We apply Lemma~\ref{lem:SRLb} with input $\eps^+$ and $\lceil 1/d^+\rceil$ to $G\big[N_G(v)\big]$, with an arbitrary equipartition into $\lceil 1/d^+\rceil$ sets as an initial partition. Note that the conditions of Lemma~\ref{lem:SRLb} are satisfied because the good event of Proposition~\ref{prop:chernoff} holds. We obtain an $(\eps,p)$-regular partition of $N_G(v)$ whose non-exceptional parts are of size between $\alpha^+pn$ and $8\alpha^+pn$, by choice of $\alpha^+$ and since $\big|N_G(v)\big|>\tfrac12pn$. If there exist $s$ parts in this partition all pairs of which form $(\eps^+,d^+,p)$-regular pairs, then these parts form the desired $Q_{v,1}$,\dots,$Q_{v,s}$. So we may assume for a contradiction that no such $s$ parts exist. It follows that when we delete all edges within parts, meeting the exceptional sets, in irregular pairs, and in pairs of density less than $d^+p$, we remove all copies of $K_s$ from $G\big[N_G(v)\big]$.
 
 The total number of such edges is, since the good event of Proposition~\ref{prop:chernoff} holds, at most
 \begin{align*}
  (d^+)^{-1}\cdot 8p^3n^2(d^+)^{2}+2p(2\eps^+pn)(2pn)+4\eps^+p^3n^2+4d^+p^3n^2&\le (12\eps^++12d^+)p^3n^2\\
  &\le 2^{-s}\gamma p^3n^2\,,
 \end{align*}
 where the final inequality is by choice of $d^+$ and $\eps^+$.
 We now estimate simply how many copies of $K_{s+1}$ a given edge $e$,
 together with $v$, can make in $\Gamma$. Since by~\eqref{eq:nolargedegs} any $\ell$-tuple of vertices of $\Gamma$ has at most $2p^\ell n$ common neighbours, the number of copies of $K_4$ containing $e$ and $v$ is at most $2p^3n$, and inductively the number of copies of $K_{s+1}$ containing $e$ and $v$ is at most
 \[\prod_{\ell=3}^{s}2p^\ell n=2^{s-2}p^{\binom{s+1}{2}-3}n^{s-2}\,.\]
 Putting these estimates together we see that the total number of copies of $K_s$ in $G\big[N_G(v)\big]$ is at most $\tfrac12\gamma p^{\binom{s+1}{2}}n^{s}$. This is the desired contradiction, completing the proof. 
\end{claimproof}

We apply Lemma~\ref{lem:G} to $G$, with input $\gamma$, $k$, $r_0$ and $\eps^-$, to obtain an integer $r$ with $10\gamma^{-1}\le kr\le r_1$, a set $V_0\subset V(G)$ with $|V_0|\le C^\ast p^{-2}$, a $k$-equitable partition $\cV=\big\{V_{i,j}\big\}_{i\in[r],j\in[k]}$ of $V(G)\setminus V_0$, and a graph $R^k_r$ on $[r]\times[k]$ with minimum degree $\delta(R^k_r)\ge\big(\tfrac{k-1}{k}+\tfrac{\gamma}{2}\big)kr$, such that $K^k_r\subset B^k_r\subset R^k_r$ and such that the following hold.
\begin{enumerate}[label=\itmarabp{G}{a}]
\item\label{main:Gsize} $\frac{n}{4kr}\leq |\Vij| \leq \frac{4n}{kr}$ for every $i\in[r]$ and $j\in[k]$,
\item\label{main:Greg} $\cV$ is $(\eps^-,d,p)_G$-regular on $R^k_r$ and $(\eps^-,d,p)_G$-super-regular on $K^k_r$,
\item\label{main:Ginh} both $\big(\NGa(v, V_{i,j}),V_{i',j'}\big)$ and $\big(\NGa(v, V_{i,j}),\NGa(v, V_{i',j'})\big)$ are $(\eps^-, d,p)_G$-regular pairs for every $\{(i,j),(i',j')\} \in E(R^k_r)$ and $v\in V\setminus V_0$, and
\item\label{main:Ggam} $|\NGa(v,V_{i,j})| = (1 \pm \eps)p|\Vij|$ for every $i \in [r]$, $j\in [k]$ and every $v \in V \setminus V_0$.
\end{enumerate}

Given $i\in[r]$, because $\delta(R^k_r)>(k-1)r$, there exists $v\in V(R^k_r)$ adjacent to each $(i,j)$ with $j\in[k]$. This, together with our assumptions on $H$, allow us to apply Lemma~\ref{lem:H2} to $H$, with input $D$, $k$, $r$, $\tfrac{1}{10}\xi$ and $\beta$, and with $m_{i,j}:=|V_{i,j}|+\tfrac{1}{kr}|V_0|$ for each $i\in[r]$ and $j\in[k]$, choosing the rounding such that the $m_{i,j}$ form a $k$-equitable integer partition of $n$. Since $\Delta(H)\le\Delta$, in particular $H$ is $\Delta$-degenerate. Let $f\colon V(H) \to [r] \times [k]$ be the mapping returned by Lemma~\ref{lem:H2}, let $W_{i,j} := f^{-1}(i,j)$, and let $X \subseteq V(H)$ be the set of special vertices returned by Lemma~\ref{lem:H2}. For every $i\in [r]$ and $j\in [k]$ we have
\begin{enumerate}[label=\itmarabp{H}{a}] 
\item\label{H:size} $m_{i,j} - \tfrac{1}{10}\xi n  \leq |W_{i,j}| \leq m_{i,j} + \tfrac{1}{10}\xi n$,
\item\label{H:sizeX} $|X| \leq \xi n$,
\item\label{H:edge} $\{f(x),f(y)\} \in E(R^k_r)$  for every $\{x,y\} \in E(H)$,
\item\label{H:special} $y,z\in \bigcup_{j'\in[k]}f^{-1}(i,j')$ for every $x\in f^{-1}(i,j)\setminus X$ and $xy,yz\in E(H)$, and
\item\label{H:v1} $f(x)=\big(1,\sigma(x)\big)$ for every $x$ in the first $\sqrt{\beta}n$ vertices of $\mathcal{L}$.
\end{enumerate}
We let $F$ be the first $\beta n$ vertices of $\mathcal{L}$. By definition of $\mathcal{L}$, in $F$ there are at least $C p^{-2}$ vertices whose neighbourhood in $H$ receives at most $s$ colours from $\sigma$.

Next, we apply Lemma~\ref{lem:hypgeo}, with input $\eps\mu^2$ and $\Delta$, to choose a set $S\subset V(G)$ of size $\mu n$. We let the $T_i$ of Lemma~\ref{lem:hypgeo} be all sets which are common neighbourhoods in $\Gamma$ of at most $\Delta$ vertices of $\Gamma$, and all sets which are common neighbourhoods in $G$ of at most $\Delta$ vertices of $\Gamma$ into any set of $\cV$, together with the sets $V_{i,j}$ for $i\in[r]$ and $j\in[k]$, and the sets $Q_{v,i}$ for $v\in V(G)$ and $i\in[s]$. We let the regular pairs $(X_i,Y_i)$ of Lemma~\ref{lem:hypgeo} be the pairs $(Q_{v,i},Q_{v,j})$ for $1\le i<j\le s$ and $v\in V(G)$, and all regular pairs $(V_{i,j},V_{i',j'})\in R^k_r$.

The result of Lemma~\ref{lem:hypgeo} is that for any $1\le\ell\le\Delta$, any $V\in\cV$, and any vertices $u_1,\dots,u_\ell$ of $V(G)$, we have
\begin{equation}\label{eq:intS}
 \begin{split}
  \Big|S\cap\bigcap_{1\le i\le\ell}N_\Gamma(u_i)\Big|&=(1\pm\eps\mu)\mu\Big|\bigcap_{1\le i\le\ell}N_\Gamma(u_i)\Big|\pm \eps\mu p^\ell n\,,\\
  \Big|S\cap V\cap \bigcap_{1\le i\le\ell}N_G(u_i)\Big|&=(1\pm\eps\mu)\mu\Big|V\cap\bigcap_{1\le i\le\ell}N_G(u_i)\Big|\pm \tfrac{\eps\mu p^\ell n}{4kr}\,,\quad \text{and}\\
  \big|S\cap V_{i,j}\big|&=(1\pm\tfrac12\eps)\mu|V_{i,j}|\quad\text{for each $i\in[r]$ and $j\in[k]$,}
 \end{split}
\end{equation}
where we use the fact $p\ge C\big(\tfrac{\log n}{n}\big)^{1/\Delta}$ and choice of $C$ to deduce $\Ca\log n<\tfrac{\eps\mu p^\Delta n}{4kr}$. Furthermore, for each $v\in V(G)$ and $1\le i<j\le s$ the pair $\big(Q_{v,i}\cap S,Q_{v,j}\cap S\big)$ is $\big(\eps^{+}_0,d^+,p)$-regular in $G$, and for each $(V_{i,j},V_{i',j'})\in R^k_r$ the pair $\big(V_{i,j}\cap U,V_{i',j'}\cap U\big)$ is $(\eps,d,p)$-regular in $G$.

Our next task is to create the pre-embedding that covers the vertices of
$V_0$. We use the following algorithm, starting with $\phi_0$ the empty
partial embedding.

\begin{algorithm}
\caption{Pre-embedding}
\label{alg:pre}
 $t:=0$ \;
 \While{$V_0\setminus\im(\phi_t)\neq\emptyset$}{
  \lnl{line:choosev} Let $v_{t+1}\in V_0\setminus\im(\phi_t)$ maximise $\big|N_G(v)\cap S\cap\im(\phi_t)\big|$ over $v\in V_0\setminus\im(\phi_t)$ \;
  Choose $x_{t+1}\in F$ such that $\big|\sigma\big(N_H(x)\big)\big|\le s$ and $\dist\big(x_{t+1},\Dom(\phi_t)\big)\ge 100k^2$ \;
  $H_{t+1}:= H \left[\big\{y\in V(H):\dist(x_{t+1},y)\le s+1\big\}\right]$ \;
  Let $G'_{t+1}$ be the maximum subgraph of
  $G\big[(S\cup\{v_{t+1}\})\setminus\im(\phi_t)\big]$ \break\mbox{}\hspace{7cm} with minimum degree $\left(\tfrac{k-1}{k}+\tfrac{\gamma}{4}\right)\mu pn$ \;
  Let $\phi$ and $q_1,\dots,q_k$ be given by Lemma~\ref{lem:coverv} with input $G'_{t+1}$, $H'_{t+1}$ and colouring $\sigma|_{V(H')}$ \;
  $\phi_{t+1}:=\phi_t\cup\phi$ \;
  \ForEach{$y\in H_{t+1}$ such that $\dist(x_{t+1},y)=s+1$}{
   Let $f^{**}(y):=\kq_{\sigma(y)}$ \;
   Let $J_y:=\phi\big(\Dom(\phi)\cap N_H(y)\big)$ \;
   Let $I'_y:=N_{G}(J_y)\cap V_{\kq_{\sigma(y)}}\cap V(G'_{t+1})$ \;
  }
  $t:=t+1$ \;
 }
\end{algorithm}

Suppose this algorithm does not fail, terminating with $t=t^*$ and with a final embedding $\phi:=\phi_{t^*}$. Let $H'=H\setminus\Dom(\phi)$. Then $\phi$ is an embedding of $H\big[V(H)\setminus V(H')\big]$ into $V(G)$ which covers $V_0$ and is contained in $V_0\cup S$. The algorithm in addition defines $f^{**}(y)\in R^k_r$, $J_y\subset S$ and $I'_y\subset S$ for each $y\in V(H')$ which has $H$-neighbours in $\Dom(\phi)$. The meanings of these are as follows. When we apply the sparse blow-up lemma, we will embed $y$ to the cluster $V_{f^{**}(y)}$. We will need to image restrict $y$ (as in Definition~\ref{def:restrict}), and the image restricting vertices will be $J_y$. The set $I'_y$ will \emph{not} be the image restriction we use, but we will deduce the existence of a suitable image restriction from $I'_y$. Before we explain this, we first claim that the algorithm does not fail, and the requirements of Lemma~\ref{lem:coverv} are met at each iteration.

\begin{claim} Algorithm~\ref{alg:pre} does not fail, and the conditions of Lemma~\ref{lem:coverv} are met at each iteration.
\end{claim}
\begin{claimproof}
Observe that in total we embed at most $\Delta^{s+2}$ vertices in each iteration, and the number of iterations is at most $|V_0|\le \Ca p^{-2}$, so that the total number of vertices we embed is at most $\Ca\Delta^{s+2}p^{-2}$.

We begin by discussing the choice of $v_{t+1}$. Suppose that at some time
$t$ we pick a vertex $v=v_{t+1}$ such that $\big|N_G(v)\cap
S\cap\im(\phi_t)\big|>\tfrac12\alpha^+\mu pn$. For each
$t-\tfrac14\Delta^{-s-2}\mu\alpha^+ pn\le t'<t$, we have $\big|N_G(v)\cap
S\cap\im(\phi_{t'})\big|>\tfrac14\alpha^+\mu pn$, yet at each of these
times $v$ is not picked, so that the vertex picked at each time $t'$ has at
least $\tfrac14\alpha^+\mu pn$ neighbours in $\im(\phi_t)\cap S$, and in
particular in $\im(\phi_t)$, a set of size at most
$\Ca\Delta^{s+2}p^{-2}$. Let $Z$ be a superset of $\im(\phi_t)$ of size at
least $\Ca p^{-1}\log n$. Now the good event of
Proposition~\ref{prop:chernoff} states that in $\Gamma$ at most $\Ca
p^{-1}\log n$ vertices of $\Gamma$ have more than
$2p|Z|<\tfrac14\alpha^+\mu pn$ neighbours in $Z$. Since
$\tfrac14\Delta^{-s-2}\mu\alpha^+ pn>\Ca p^{-1}\log n$ by choice of $p$, this is a contradiction. We conclude that at each time $t$, the vertex $v_{t+1}$ picked at time $t$ satisfies $\big|N_G(v)\cap S\cap\im(\phi_t)\big|\le\tfrac12\alpha^+\mu pn$.

From this point on we consider a fixed time $t$, and write $v$ rather than $v_{t+1}$, and $\phi$ for $\phi_t$, and so on.

Since we cover at most $\Ca\Delta^{s+2} p^{-2}$ vertices, so we have $|S\setminus\im(\phi)|=(1\pm\tfrac12\eps)\mu n$. Now, to obtain the maximum subgraph of $G\big[(S\cup\{v\})\setminus\im(\phi)\big]$ with minimum degree $\big(\tfrac{k-1}{k}+\tfrac{\gamma}{4}\big)\mu pn$, we successively remove vertices whose degree is too small until no further remain. We claim that less than $\tfrac18\mu\alpha^+pn$ vertices are removed, and $v$ is not one of the vertices removed. To see this, observe that every vertex has at least $\big(\tfrac{k-1}{k}+\tfrac{\gamma}{2}\big)\mu pn$ neighbours in $S$ by~\eqref{eq:intS}. Suppose for a contradiction that there is a set $Z$ of $\tfrac18\mu\alpha^+pn$ vertices which are the first removed from $S$ in this process. Then each vertex of $Z$ has at least $\tfrac14\gamma\mu p n$ neighbours in $Z\cup\im(\phi)$, which by choice of $\alpha^+$ is a contradiction to the good event of Proposition~\ref{prop:chernoff}.

We conclude $\big|(S\cup\{v\})\setminus\im(\phi)\big|=(1\pm\eps)\mu
n$. Since $v$ has at least $\big(\tfrac{k-1}{k}+\tfrac{\gamma}{2}\big)\mu
pn$ neighbours in $S$, of which at most $\tfrac12\alpha^+\mu pn$ are in
$\im(\phi)$ and at most $|Z|$ are in $Z$, the vertex $v$ is not removed. Furthermore,
for each $i\in[s]$ we have $|Q_{v,i}\cap
V(G')\big|\ge\tfrac12\big|Q_{v,i}\cap S\big|$. We now use this to count
copies of $K_s$ in $N_{G'}(v)$. We choose for $i=1,\ldots,s$ sequentially
vertices in $Q_{v,i}\cap V(G')$, at each step choosing a vertex $w_i$ which is
adjacent to the previous vertices, and which is such that $w_1,\ldots,w_i$
have at least
$(d^+-\eps^+_{s-2})^ip^i|Q_{v,j}|$ common $G$-neighbours in each $Q_{v,j}$ for
$j>i$, and have $(1\pm\eps)^ip^i|Q_{v,j}|$ common $\Gamma$-neighbours in each $Q_{v,j}$ for
$j>i$, and the pair
\[\Big(\bigcap_{\ell\in[i]} N_{\Gamma}(w_\ell,Q_{v,j}), \bigcap_{\ell\in[i]} N_{\Gamma}(w_\ell,Q_{v,j'})\Big)\]
is $(\eps^+_i,d^+,p)$-regular in $G$ for each $i<j<j'\le s$. Note that all
these properties hold when $i=0$ vertices have been chosen. Assuming these
properties hold when we come to choose $w_i$, there are at least
$2^{1-i}(d^+)^{i-1}p^{i-1}|Q_{v,i}|$ vertices of $Q_{v,i}$ which are
adjacent to all previously chosen vertices. If $i=s$ then all of these are
valid choices. If $i<s$, by Propositions~\ref{prop:neighbourhood} and~\ref{prop:subpairs},
and because the good event of Proposition~\ref{prop:chernoff} holds, at most
\begin{equation*}
 s\cdot 4^i(d^+)^{1-i}\eps^+_{s-2}p^{i-1}|Q_{v,i}|+s\cdot\Ca p^{-1}\log n
\end{equation*}
vertices of $Q_{v,i}$ cause the numbers of $G$- or $\Gamma$-common
neighbours in some $Q_{v,j}$ for $j>i$ to go wrong. Finally, if $i=s-1$
then there is no choice of $i<j<j'\le s$ and so no failure of
regularity can occur, while if $i<s-1$ then by the good event of
Lemma~\ref{lem:TSRIL} the number of vertices which cause a failure of
regularity is at most $s^2\Ca p^{-2}\log n$. By choice of $\eps^+_{s-2}$
and $p$, in total at least $2^{-i}(d^+)^{i-1}p^{i-1}|Q_{v,i}|$ vertices of
$Q_{v,i}$ are thus valid choices for $w_i$. Finally, by choice of
$\gamma^+$ the total number of copies of $K_s$ in $N_{G'}(v)$ is at least
$2\gamma^+ p^{\binom{s}{2}}\big(p|S|\big)^s \ge \gamma^+ p^{\binom{s+1}{2}}(\mu n)^s$, as desired.

The remaining conditions of Lemma~\ref{lem:coverv} are simpler to check. By~\eqref{eq:intS} we have $\big|N_{G'}(W)\big| \le \big|N_\Gamma(W)\cap S\big|\le 2\mu n p^{|W|}$ for any $W \subset V(G')$ of size at most $\Delta$. The graph $G$ with the regular partition $(V_{i,j})_{i\in[r],j\in[k]}$, with reduced graph $R^k_r$, has the required minimum degree. By~\eqref{eq:intS} the intersection of the part $V_{i,j}$ with $S$ has size $(1\pm\tfrac12 \eps)\mu |V_{i,j}|$, so that $|V_{i,j}\cap V(G')|=(1\pm\eps)\mu|V_{i,j}|$ as required. Furthermore the regular pairs of $R$ intersected with $S$ are regular, and so by Proposition~\ref{prop:subpairs} the subpairs obtained by intersecting with $V(G')$ (which is, except for $v$, contained in $S$; and $v$ is in $V_0$ hence not in any of these pairs) are also sufficiently regular. Finally, the graph $H_{t+1}$ chosen at each time $t$ satisfies the conditions of Lemma~\ref{lem:coverv} by definition. Note that we can at each step choose $x_{t+1}$ and hence $H_{t+1}$ because there are at least $Cp^{-2}$ vertices of $F$ whose neighbourhood is coloured with at most $s$ colours; even after embedding all of $V_0$, the domain of $\phi$ contains at most $\Ca\Delta^{s+2}p^{-2}$ vertices, and hence at most $\Ca\Delta^{s+100k^2+3}p^{-2}<Cp^{-2}$ vertices of $H$ are too close to $\Dom(\phi)$.
\end{claimproof}

 We next define
 image restricting vertex sets and create an updated homomorphism
 $f^*:V(H')\to [r]\times[k]$. The former is easier. Let $X^{**}$ consist of the vertices of $H'$ which have at least one $H$-neighbour in $\Dom(\phi)$. The vertices of
 $\Dom(\phi)$ are partitioned according to the $x_t$ chosen at each time in
 Algorithm~\ref{alg:pre}, and because these vertices are chosen far apart
 in $H$, any vertex $y$ of $X^{**}$ is at
 distance $s+1$ from some $x_t$. The neighbours in $H'$ of $y$ are either also at
 distance $s+1$ in $H$ from $x_t$ and not adjacent to any vertices of
 $\Dom(\phi)$ corresponding to other $x_{t'}$, or they are not adjacent to
 any vertex of $\Dom(\phi)$ at all. It follows that for each $y\in X^{**}$ the quantities $f^{**}(y)$, $J_y$ and $I'_y$ are set exactly once in the running of Algorithm~\ref{alg:pre}. By Lemma~\ref{lem:coverv} and~\eqref{eq:intS}, given $y\in X^{**}$, we have $|I'_y|\ge 2\alpha p^{|J_y|}|(1-\eps)\mu|V_{f^{**}(y)}|$. We claim this implies
 \begin{equation}\label{eq:mainproof:sizeI}
  \big|N_G(J_y)\cap V_{f^{**}(y)}\big|\ge\alpha p^{|J_y|}\big|V_{f^{**}(y)}\big|\,.
 \end{equation}
 Indeed, suppose for a contradiction that~\eqref{eq:mainproof:sizeI} fails. Since $I'_y$ is by construction contained in $S$, we have $|I'_y|\le\big|N_G(J_y)\cap V_{f^{**}(y)}\cap S\big|$. Using~\eqref{eq:intS} to estimate the size of the latter set, we get
 \[|I'_y|\le (1\pm\eps\mu)\mu\cdot\alpha p^{|J_y|}\big|V_{f^{**}(y)}\big|+\tfrac{\eps\mu p^{|J_y|}n}{4kr}<2\alpha p^{|J_y|}|(1-\eps)\mu|V_{f^{**}(y)}|\,,\]
 where the final inequality is by choice of $\eps$ and since $|V_{f^{**}(y)}|\ge\tfrac{n}{4kr}$ by~\ref{main:Gsize}. This is in contradiction to the lower bound on $|I'_y|$ from Lemma~\ref{lem:coverv} stated above.

We construct the updated homomorphism as follows. We will have $f^*(y)=f(y)$
for all vertices which are not within distance $s+\binom{k+1}{2}$ of $\Dom(\phi)$ in
$H$. Given a vertex $x$ of $H$ chosen at some time $t$ in
Algorithm~\ref{alg:pre}, we set $f^*(y)$ for each $y$ at distance between $s+1$
and $s+\binom{k+1}{2}$ from $x$ in $H$ as follows. We will generate a collection
$Z_1,\ldots,Z_{\binom{k+1}{2}}$ of copies of $K_k$ in $R^k_r$, each labelled with the
integers $1,\ldots,k$. For each $i=1,\ldots,\binom{k+1}{2}$, if $y$ is at distance
$s+i$ from $x$ in $H$, then we set $f^*(y)$ to be the label $\sigma(y)$
cluster of $Z_i$. The properties of the sequence $Z_1,\ldots,Z_{\binom{k+1}{2}}$ we
require are the following. First, $Z_1$ is the clique returned by the
application of Lemma~\ref{lem:coverv} at $x$ with the labelling given by
that lemma. Second, $Z_{\binom{k+1}{2}}$ is the clique
$\big(V_{1,1},\dots,V_{1,k}\big)$, labelled $1,\ldots,k$ in that
order. Third, for each $i=2,\ldots,\binom{k+1}{2}$, each cluster of $Z_i$ is adjacent in $R^k_r$ to each
differently-labelled cluster of $Z_{i-1}$. Assuming such a sequence of cliques exists, the resulting $f^*$ has
the properties that each vertex $y$ of $X^{**}$ is assigned by
$f^*$ to $f^{**}(y)$, that each edge of $H'$ is mapped by $f^*$ to an edge
of $R^k_r$, and that $f$ and $f^*$ disagree on at most $\Ca p^{-2}\Delta^{s+\binom{k+1}{2}+3}$ vertices of $H'$, all in the first
$\sqrt{\beta}n$ vertices of $\cL$. These will be the properties we need of
$f^*$. Note that this definition is consistent, in that it does not attempt
to set $f^*(y)$ to two different clusters for any $y$, because the vertices
chosen at each step of Algorithm~\ref{alg:pre} are at pairwise distance at
least $100k^2$. It remains only to show that the desired sequence of cliques
always exists.
\begin{claim} For any $k$-cliques $Z_1$ and $Z_{\binom{k+1}{2}}$ in $R^k_r$ a sequence
  $Z_1,\ldots,Z_{\binom{k+1}{2}}$ with the above properties exists.
\end{claim}
\begin{claimproof}
By the minimum degree of $R^k_r$, any $k$-set in $V(R^k_r)$ has at least
one common neighbour. We will use this fact at each step in the following
algorithm. Set $t=2$. We loop through
$j=1,\ldots,k-1$ sequentially. For each value of $j$ we perform the following operation.

For each
$i=j+1,\ldots,k$ sequentially, choose a cluster $w_{t}$ of $R^k_r$ which is adjacent
to all the clusters of $Z_{t-1}$ except possibly that labelled $i$, and
which is also adjacent to the cluster of $Z_{\binom{k+1}{2}}$ labelled $j$. We let
$Z_{t}$ be the clique obtained from $Z_{t-1}$ by replacing the label $i$
cluster with $w_t$, which we label $i$; all other clusters keep their
previous label. We increment $t$.

After performing the $i=k$ operation, we let $Z_{t}$ be obtained from $Z_{t-1}$ by
replacing the label $j$ cluster of $Z_{t-1}$ with the label $j$ cluster of
$Z_{\binom{k+1}{2}}$, and increment $t$. We now proceed with the next round of the
$j$-loop.

Observe that after the completion of each $j$-loop, the clusters of
$Z_{t-1}$ labelled $1,\ldots,j$ are the same as those of
$Z_{\binom{k+1}{2}}$. In particular the given $Z_{\binom{k+1}{2}}$ has the
required adjacencies in $Z_{\binom{k+1}{2}-1}$ (the final clique
constructed in the $j=k-1$ loop), while the remaining
required adjacencies hold by construction.
\end{claimproof}

At this point we complete the proof almost exactly as in~\cite{ABET}. What
follows is taken from there, with only trivial changes, for completeness' sake.

 For each $i\in[r]$ and $j\in[k]$, let $W'_{i,j}$ be the set of vertices $w\in V(H')$ with $f^*(w)\in V_{i,j}$, and let $X'$ consist of $X$ together with all vertices of $H'$ at $H$-distance $100k^2$ or less from some $x_t$ with $t\in[t^*]$. The total number of vertices $z\in V(H)$ at distance at most $100k^2$ from some $x_t$ is at most $2\Delta^{200k^2}|V_0|<\tfrac{1}{100}\xi n$. Since $W_{i,j}\symd W'_{i,j}$ contains only such vertices, we have
\begin{enumerate}[label=\itmarabp{H}{b}]
 \item\label{Hp:sizeWp} $m_{i,j}-\tfrac15\xi n\le |W'_{i,j}|\le m_{i,j}+\tfrac15\xi n$,
 \item\label{Hp:sizeX} $|X'| \leq 2\xi n$, 
 \item\label{Hp:edge} $\{f^*(x),f^*(y)\} \in E(R^k_r)$  for every $\{x,y\} \in E(H')$, and
 \item\label{Hp:special} $y,z\in \bigcup_{j'\in[k]}W'_{i,j'}$ for every $x\in W'_{i,j}\setminus X'$ and $xy,yz\in E(H')$.
\end{enumerate}
where~\ref{Hp:sizeX},~\ref{Hp:edge} and~\ref{Hp:special} hold by~\ref{H:sizeX} and definition of $X'$, by definition of $f^*$, and by~\ref{H:special} and choice of $X'$ respectively.

Furthermore, we have
\begin{enumerate}[label=\itmarabp{G}{a}]
\item $\frac{n}{4kr}\leq |\Vij| \leq \frac{4n}{kr}$ for every $i\in[r]$ and $j\in[k]$,
\item $\cV$ is $(\eps,d,p)_G$-regular on $R^k_r$ and $(\eps,d,p)_G$-super-regular on $K^k_r$,
\item both $\big(\NGa(v, V_{i,j}),V_{i',j'}\big)$ and $\big(\NGa(v, V_{i,j}),\NGa(v, V_{i',j'})\big)$ are $(\eps, d,p)_G$-regular pairs for every $\{(i,j),(i',j')\} \in E(R^k_r)$ and $v\in V\setminus V_0$, and
\item $|\NGa(v,V_{i,j})| = (1\pm \eps)p|\Vij|$ for every $i \in [r]$, $j\in [k]$ and every $v \in V \setminus V_0$.
 \item\label{main:GpI} $\big|V_{f^*(x)}\cap\bigcap_{u\in J_x}N_G(u)\big|\ge\alpha p^{|J_x|}|V_{f^*(x)}|$ for each $x\in V(H')$,
 \item\label{main:GpGI} $\big|V_{f^*(x)}\cap\bigcap_{u\in J_x}N_\Gamma(u)\big|=(1\pm\eps^*)p^{|J_x|}|V_{f^*(x)}|$  for each $x\in V(H')$, and
 \item\label{main:GpIreg} $\big(V_{f^*(x)}\cap\bigcap_{u\in J_x}N_\Gamma(u),V_{f^*(y)}\cap\bigcap_{v\in J_y}N_\Gamma(v)\big)$ is $(\eps^*,d,p)_G$-regular for each $xy\in E(H')$.
 \item\label{main:GaI} $\big|\bigcap_{u\in J_x}N_\Gamma(u)\big|\le(1+\epsa) p^{|J_x|}n$ for each $x\in V(H')$,
\end{enumerate}
Properties~\ref{main:Gsize} to~\ref{main:Ggam} are repeated for convenience
(replacing $\eps^-$ with the larger $\eps$). Properties~\ref{main:GpI},~\ref{main:GpGI} and~\ref{main:GaI}, are trivial when $J_x=\emptyset$. Otherwise,~\ref{main:GpI} is guaranteed by~\eqref{eq:mainproof:sizeI}, and~\ref{main:GpGI} and~\ref{main:GaI} are guaranteed by Lemma~\ref{lem:coverv}. Finally~\ref{main:GpIreg} follows from~\ref{main:Greg} when $J_x,J_y=\emptyset$, and otherwise is guaranteed by Lemma~\ref{lem:coverv}, as follows. If both $J_x$ and $J_y$ are non-empty, then~\ref{item:pel-tsril} states that the desired pair is $(\epsa,d,p)_G$-regular. If $J_x$ is empty and $J_y$ is not, then necessarily $|J_x|\le\Delta-1$, and by~\ref{item:pel-osril} the pair $\big(V_{f^*(x)}\cap\bigcap_{u\in J_x}N_\Gamma(u),V_{f^*(y)}\big)$ is $(\eps^*,d,p)_G$-regular.

For each $i\in[r]$ and $j\in[k]$, let $V'_{i,j}=V_{i,j}\setminus\im(\phi_{t^*})$, and let $\cV'=\{V'_{i,j}\}_{i\in[r],j\in[k]}$. Because $V_{i,j}\setminus V'_{i,j}\subset S$ for each $i\in[r]$ and $j\in[k]$, using~\eqref{eq:intS} and Proposition~\ref{prop:subpairs}, and our choice of $\mu$, we obtain
\begin{enumerate}[label=\itmarabp{G}{b}]
\item\label{Gp:sizeV} $\frac{n}{6kr}\leq |V'_{i,j}| \leq \frac{6n}{kr}$ for every $i\in[r]$ and $j\in[k]$,
\item\label{Gp:Greg} $\cV'$ is $(2\eps,d,p)_G$-regular on $R^k_r$ and $(2\eps,d,p)_G$-super-regular on $K^k_r$,
\item\label{Gp:Ginh} both $\big(\NGa(v, V'_{i,j}),V'_{i',j'}\big)$ and $\big(\NGa(v, V'_{i,j}),\NGa(v, V'_{i',j'})\big)$ are $(2\eps, d,p)_G$-regular pairs for every $\{(i,j),(i',j')\} \in E(R^k_r)$ and $v\in V\setminus V_0$, and
\item\label{Gp:GsGa} $|\NGa(v,V'_{i,j})| = (1 \pm 2\eps)p|V_{i,j}|$ for every $i \in [r]$, $j\in [k]$ and every $v \in V \setminus V_0$.
 \item\label{Gp:sizeI} $\big|V'_{f^*(x)}\cap\bigcap_{u\in J_x}N_G(u)\big|\ge\tfrac12\alpha p^{|J_x|}|V'_{f^*(x)}|$,
 \item\label{Gp:sizeGa} $\big|V'_{f^*(x)}\cap\bigcap_{u\in J_x}N_\Gamma(u)\big|=(1\pm2\eps^*)p^{|J_x|}|V'_{f^*(x)}|$, and
 \item\label{Gp:Ireg} $\big(V'_{f^*(x)}\cap\bigcap_{u\in J_x}N_\Gamma(u),V'_{f^*(y)}\cap\bigcap_{v\in J_y}N_\Gamma(v)\big)$ is $(2\eps^*,d,p)_G$-regular.
 \item\label{Gp:GaI} $\big|\bigcap_{u\in J_x}N_\Gamma(u)\big|\le(1+2\epsa) p^{|J_x|}n$ for each $x\in V(H')$,
\end{enumerate}

 We are now almost finished. The only remaining problem is that we do not necessarily have $|W'_{i,j}|=|V'_{i,j}|$ for each $i\in[r]$ and $j\in[k]$. Since $|V'_{i,j}|=|V_{i,j}|\pm 2\Delta^{200k^2}|V_0|=m_{i,j}\pm 3\Delta^{200k^2}|V_0|$, by~\ref{Hp:sizeWp} we have $|V'_{i,j}|=|W'_{i,j}|\pm \xi n$. We can thus apply Lemma~\ref{lem:balancing}, with input $k$, $r_1$, $\Delta$, $\gamma$, $d$, $8\eps$, and $r$. This gives us sets $V''_{i,j}$ with $|V''_{i,j}|=|W'_{i,j}|$ for each $i\in[r]$ and $j\in[k]$ by~\ref{lembalancing:sizesout}. Let $\cV''=\{V''_{i,j}\}_{i\in[r],j\in[k]}$. Lemma~\ref{lem:balancing} guarantees us the following.
\begin{enumerate}[label=\itmarabp{G}{c}]
\item\label{Gpp:sizeV} $\frac{n}{8kr}\leq |V''_{i,j}| \leq \frac{8n}{kr}$ for every $i\in[r]$ and $j\in[k]$,
\item\label{Gpp:Greg} $\cV''$ is $(4\epsa,d,p)_G$-regular on $R^k_r$ and $(4\epsa,d,p)_G$-super-regular on $K^k_r$,
\item\label{Gpp:Ginh} both $\big(\NGa(v, V''_{i,j}),V''_{i',j'}\big)$ and $\big(\NGa(v, V''_{i,j}),\NGa(v, V''_{i',j'})\big)$ are $(4\epsa, d,p)_G$-regular pairs for every $\{(i,j),(i',j')\} \in E(R^k_r)$ and $v\in V\setminus V_0$, and
\item\label{Gpp:GsGa} we have 
 $(1-4\eps)p|V''_{i,j}| \leq |\NGa(v,V''_{i,j})| \leq (1 + 4\eps)p|V''_{i,j}|$ for every $i \in [r]$, $j\in [k]$ and every $v \in V \setminus V_0$.
 \item\label{Gpp:sizeI} $\big|V''_{f^*(x)}\cap\bigcap_{u\in J_x}N_G(u)\big|\ge\tfrac14\alpha p^{|J_x|}|V''_{f^*(x)}|$,
 \item\label{Gpp:sizeGa} $\big|V''_{f^*(x)}\cap\bigcap_{u\in J_x}N_\Gamma(u)\big|=(1\pm4\eps^*)p^{|J_x|}|V'_{f^*(x)}|$, and
 \item\label{Gpp:Ireg} $\big(V''_{f^*(x)}\cap\bigcap_{u\in J_x}N_\Gamma(u),V''_{f^*(y)}\cap\bigcap_{v\in J_y}N_\Gamma(v)\big)$ is $(4\eps^*,d,p)_G$-regular.
\end{enumerate} 
Here~\ref{Gpp:sizeV} comes from~\ref{Gp:sizeV} and~\ref{lembalancing:symd}, while~\ref{Gpp:Greg} comes from~\ref{lembalancing:regular} and choice of $\eps$. \ref{Gpp:Ginh} is guaranteed by~\ref{lembalancing:inheritance}. Now, each of~\ref{Gpp:GsGa},~\ref{Gpp:sizeI} and~\ref{Gpp:sizeGa} comes from the corresponding~\ref{Gp:GsGa},~\ref{Gp:sizeI} and~\ref{Gp:sizeGa} together with~\ref{lembalancing:gammaout}. Finally,~\ref{Gpp:Ireg} comes from~\ref{Gp:Ireg} and~\ref{Gp:GaI} together with Proposition~\ref{prop:subpairs} and~\ref{lembalancing:gammaout}.

For each $x\in V(H')$ with $J_x=\emptyset$, let $I_x=V''_{f^*(x)}$. For each $x\in V(H')$ with $J_x\neq\emptyset$, let $I_x=V''_{f^*(x)}\cap\bigcap_{u\in J_x}N_G(u)$. Now $\cW'$ and $\cV''$ are $\kappa$-balanced by~\ref{Gpp:sizeV}, size-compatible by construction, partitions of respectively $V(H')$ and $V(G)\setminus\im(\phi_{t^*})$, with parts of size at least $n/(\kappa r_1)$ by~\ref{Gpp:sizeV}. Letting $\widetilde{W}_{i,j}:=W'_{i,j}\setminus X'$, by~\ref{Hp:sizeX}, choice of $\xi$, and~\ref{Hp:special}, $\{\widetilde{W}_{i,j}\}_{i\in[r],j\in[k]}$ is a $\big(\vartheta,K^k_r\big)$-buffer for $H'$. Furthermore since $f^*$ is a graph homomorphism from $H'$ to $R^k_r$, we have~\ref{itm:blowup:H}. By~\ref{Gpp:Greg},~\ref{Gpp:Ginh} and~\ref{Gpp:GsGa} we have~\ref{itm:blowup:G}, with $R=R^k_r$ and $R'=K^k_r$. Finally, the pair $(\cI,\cJ)=\big(\{I_x\}_{x\in V(H')},\{J_x\}_{x\in V(H')}\big)$ form a $\big(\rho,\tfrac14\alpha,\Delta,\Delta\big)$-restriction pair. To see this, observe that the total number of image restricted vertices in $H'$ is at most $\Delta^2|V_0|<\rho|V_{i,j}|$ for any $i\in[r]$ and $j\in[k]$, giving~\ref{itm:restrict:numres}. Since for each $x\in V(H')$ we have $|J_x|+\deg_{H'}(x)=\deg_H(x)\le\Delta$ we have~\ref{itm:restrict:Jx}, while~\ref{itm:restrict:sizeIx} follows from~\ref{Gpp:sizeI}, and~\ref{itm:restrict:sizeGa} follows from~\ref{Gpp:sizeGa}. Finally,~\ref{itm:restrict:Ireg} follows from~\ref{Gpp:Ireg}, and~\ref{itm:restrict:DJ} follows since $\Delta(H)\le\Delta$. Together this gives~\ref{itm:blowup:restrict}. Thus, by Lemma~\ref{thm:blowup} there exists an embedding $\phi$ of $H'$ into $G\setminus\im(\phi_{t^*})$, such that $\phi(x)\in I_x$ for each $x\in V(H')$. Finally, $\phi\cup\phi_{t^*}$ is an embedding of $H$ in $G$, as desired.
\end{proof}

\section{Concluding remarks}
\label{sec:remarks}
\subsection{Optimality of Theorem~\ref{thm:main}}
In Theorems~\ref{thm:abet} and~\ref{thm:main}, the requirement for $C^\ast
p^{-2}$ vertices in $H$ whose neighbourhood contains few colours is optimal
up to the value of $C^\ast$. However the value of $C^\ast$ we obtain
derives from (multiple applications of) the sparse regularity lemma and is
hence very far from optimal. One can use the methods of this
paper to obtain an improved (but still far from sharp) constant, and we expect that one can use the
methods of this paper to determine an optimal $C^\ast$ asymptotically, at
least for special cases.

The way to obtain this improvement is the following. We work exactly as in
the proof of Theorem~\ref{thm:maink}, except that for each $v\in V(G)$ we
identify the largest $1\le s\le k-1$ for which there are many copies of $K_s$
in $N_G(v)$, and obtain a robust witness for this property as in that
proof. Now when we come to cover the vertices of the set $V_0$ returned by
Lemma~\ref{lem:G}, we use vertices from zero-free regions of $\cL$ which
are not in the first few vertices of $\cL$ whenever possible: in particular
this is always possible when we are to cover a vertex which is in many
copies of $K_k$. Our proof, with trivial modification, shows
that this pre-embedding method succeeds. The result is that we can reduce
$\Ca$ to a quantity on the order of $\Delta^{100k^2}$; this number comes
from our requirement to choose vertices in $\cL$ which are widely separated
in $H$ for the pre-embedding onto the vertices of $V_0$ which are not in many copies of $K_{k}$.

When $H$ contains many isolated vertices, this requirement disappears and
we can further improve. We believe (but have not attempted to prove) that there is some $C_k$ with the following property. Let $\Gamma$ be a typical instance of $G(n,p)$, where $p\gg n^{-1/k}$. Suppose $G\subset\Gamma$ has minimum degree $\big(\tfrac{k-1}{k}+o(1)\big)pn$. Then any choice of $G$ contains at most $\big(C_k+o(1)\big)p^{-2}$ vertices which are in $o\big( p^{\binom{k}{2}}n^{k-1}\big)$ copies of $K_k$; on the other hand there is a choice of $G$ which has $\big(C_k-o(1)\big)p^{-2}$ vertices not in any copy of $K_k$.

Assuming the above statement to be true, it follows that $C_k$ is the
asymptotically optimal $C^\ast$ whenever all vertices of $H$ are either
isolated or contained in a copy of $K_k$; for example when $H$ consists of
a $(k-1)$st power of a cycle together with some isolated vertices. Further
generalisation to (for example) try to establish an optimal value of
$C^\ast$ in Theorem~\ref{thm:abet} would be possible; but it would also
presumably depend on the graph structure of $H$. If the vertices of $H$
which are not in triangles are far apart in $H$, then the generalisation is
easy (and the answer is the same) but if they are not generally far apart
it seems likely that one would have to use several such vertices to cover
one badly-behaved vertex of $G$, and hence $C^\ast$ would need to be larger
than the above $C_k$.

\subsection{Local colourings of~$H$ versus global colourings}
\label{sec:construct}
Recall that Theorem~\ref{thm:abet} requires some vertices in $H$ to have neighbourhoods which contain no edges, and that this is necessary because otherwise we can `locally' avoid $H$-containment simply by picking a vertex of $G(n,p)$ and removing all edges in its neighbourhood to form $G$. Theorem~\ref{thm:main} implies that, when $H$ is $3$-colourable, this is really the only obstruction: if we insist that every vertex of $G$ has a reasonable number of edges in its neighbourhood, then $G$ contains all $3$-colourable $H$ with small bandwidth and maximum degree.

It is natural to guess that a similar `local' obstruction generalises:
perhaps for every $k$, if $H$ is a $k$-colourable graph with small
bandwidth and constant maximum degree which has $\Omega(p^{-2})$ vertices whose neighbourhoods are bipartite, then $H$ is guaranteed to be contained in any subgraph $G$ of $G(n,p)$ with sufficiently high minimum degree and in which every vertex neighbourhood has a reasonable number of edges.

The purpose of this section is to observe that the above guess is
false. Indeed, one cannot merely consider the chromatic number of vertex
neighbourhoods, but really has to take into account
the number of colours used on vertex neighbourhoods in the whole $k$-colouring of $H$ (as in the statement of Theorem~\ref{thm:main}).

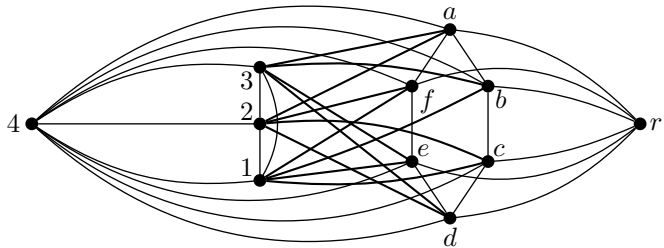
\begin{figure}[t]
\tikzstyle{every node}=[circle, draw, fill=black,
                        inner sep=0pt, minimum width=4pt]
\begin{tikzpicture}[thick,scale=1]
    \node at (-2,0) (4) [label=left:$4$] {} ;
    \node at (1,-0.75) (1) [label=above left:$1$] {} ;
    \node at (1,0) (2) [label=above left:$2$] {} ;
    \node at (1,0.75) (3) [label=below left:$3$] {} ;
    \node at (3.5,1.25) (a) [label=above:$a$] {} ;
    \node at (4,0.5) (b) [label=below right:$b$] {} ;
    \node at (4,-0.5) (c) [label=above right:$c$] {} ;
    \node at (3.5,-1.25) (d) [label=below:$d$] {} ;
    \node at (3,-0.5) (e) [label=above right:$e$] {} ;
    \node at (3,0.5) (f) [label=below right:$f$] {} ;
    \node at (6,0) (r) [label=right:$r$] {} ;
    \path[draw] (4) edge [line width=0.5pt,bend right=20] (1.center) ;
    \path[draw] (4) edge [line width=0.5pt] (2.center) ;
    \path[draw] (4) edge [line width=0.5pt,bend left=20] (3.center) ;
    \path[draw] (1) edge [line width=0.5pt] (2.center) ;
    \path[draw] (2) edge [line width=0.5pt] (3.center) ;
    \path[draw] (3) edge [line width=0.5pt,bend left] (1.center) ;
    \path[draw] (a) edge [line width=0.5pt] (b.center) ;
    \path[draw] (b) edge [line width=0.5pt] (c.center) ;
    \path[draw] (c) edge [line width=0.5pt] (d.center) ;
    \path[draw] (d) edge [line width=0.5pt] (e.center) ;
    \path[draw] (e) edge [line width=0.5pt] (f.center) ;
    \path[draw] (f) edge [line width=0.5pt] (a.center) ;
    \path[draw] (r) edge [line width=0.5pt,bend right=20] (a.center) ;
    \path[draw] (r) edge [line width=0.5pt,bend right=10] (b.center) ;
    \path[draw] (r) edge [line width=0.5pt,bend left=10] (c.center) ;
    \path[draw] (r) edge [line width=0.5pt,bend left=20] (d.center) ;
    \path[draw] (r) edge [line width=0.5pt,bend left] (e.center) ;
    \path[draw] (r) edge [line width=0.5pt,bend right] (f.center) ;
    \path[draw] (4) edge [line width=0.5pt,bend left] (a.center) ;
    \path[draw] (4) edge [line width=0.5pt,bend left=35] (b.center) ;
    \path[draw] (4) edge [line width=0.5pt,bend right=35] (c.center) ;
    \path[draw] (4) edge [line width=0.5pt,bend right] (d.center) ;
    \path[draw] (4) edge [line width=0.5pt,bend right] (e.center) ;
    \path[draw] (4) edge [line width=0.5pt,bend left] (f.center) ;
    \path[draw] (1) edge [line width=0.8pt,bend right=5] (b.center) ;
    \path[draw] (1) edge [line width=0.8pt,bend right=10] (c.center) ;
    \path[draw] (1) edge [line width=0.8pt] (e.center) ;
    \path[draw] (1) edge [line width=0.8pt] (f.center) ;
    \path[draw] (2) edge [line width=0.8pt] (a.center) ;
    \path[draw] (2) edge [line width=0.8pt,bend left=15] (c.center) ;
    \path[draw] (2) edge [line width=0.8pt] (d.center) ;
    \path[draw] (2) edge [line width=0.8pt] (f.center) ;
    \path[draw] (3) edge [line width=0.8pt] (a.center) ;
    \path[draw] (3) edge [line width=0.8pt,bend left=10] (b.center) ;
    \path[draw] (3) edge [line width=0.8pt] (d.center) ;
    \path[draw] (3) edge [line width=0.8pt] (e.center) ;
\end{tikzpicture}
\caption{The graph $F$}\label{fig:date}
\end{figure}

Consider the following graph $F$ (see Figure~\ref{fig:date}). We begin with vertices $1,2,3,4$ which form a clique, and vertices $a,b,c,d,e,f$ which form a cycle of length six (in that order). We join $4$ to all of $a,b,c,d,e,f$, we join $1$ to $b,c,e,f$, and $2$ to $a,c,d,f$, and $3$ to $a,b,d,e$. Finally we add a vertex $r$ adjacent to $a,b,c,d,e,f$.

This graph has the following properties. It is $4$-colourable and in any
$4$-colouring the vertices $a,d$ have the same colour as $1$, the vertices $b,e$ have the same colour as $2$, and $c,f$ have the same colour as $3$. All vertices except $r$ are in a copy of $K_4$. The neighbourhood of $r$ is a cycle of order $6$, which is bipartite.

Given $n$ divisible by $11$, we let $H$ consist of $n/11$ disjoint copies of $F$. By Theorem~\ref{thm:main}, with $s=3$, if $G$ is a subgraph of a typical random graph $\Gamma=G(n,p)$, where $p\gg\big(\tfrac{\log n}{n}\big)^{1/9}$, such that $\delta(G)\ge\big(\tfrac{3}{4}+\gamma\big)pn$, and in addition the neighbourhood of every vertex of $G$ contains at least $\gamma p^9n^3$ copies of $K_3$, then we have $H\subset G$.

Observe that we cannot take $s$ smaller than $3$, since in every $4$-colouring of $F$ every vertex has three different colours in its neighbourhood, including $r$. This is why Theorem~\ref{thm:abet} requires many copies of $K_3$ in every vertex neighbourhood. However the neighbourhood of $r$ itself is $K_3$-free, and in fact bipartite. We now give a construction that shows that it is not enough for every vertex neighbourhood to contain many edges (or indeed many copies of $C_6$).

We begin by selecting (for some small $\eps>0$) a set $X$ of $\eps p^{-1}$ vertices, and then generating $\Gamma=G(n,p)$. With high probability no vertex of $X$ has more than $\log n$ neighbours in $X$, and the joint neighbourhood $Y$ of $X$ has size at most $2\eps n$. We randomly partition $Y=Y_1\cup Y_2$ into two equal parts, and we randomly partition $Z:=V(\Gamma)\setminus(X\cup Y)$ into five equal parts $Z_1,\dots,Z_5$.

We let $G$ be the subgraph of $\Gamma$ obtained by taking all edges from $X$ to $Y$, all edges between $Y_1$ and $Y_2$, all edges from $Y_1$ to $Z\setminus Z_1$ and from $Y_2$ to $Z\setminus Z_2$, and all edges within $Z$ which are not contained in any $Z_i$. It is easy to check that with high probability $G$ has minimum degree roughly $\tfrac45pn$, and that neighbourhoods of all vertices contain many edges (and many copies of $C_6$).

However we claim $G$ does not contain $H$. Indeed, consider any $x\in X$. Since $N_G(x)$ is contained in $Y$, the graph $G\big[N_G(x)\big]$ is bipartite, so that any copy of $F$ using $x$ must place $r\in F$ on $x$. Furthermore, the vertices $a,b,c,d,e,f$ must be placed alternating in $Y_1$ and $Y_2$. Without loss of generality suppose $a,c,e\in Y_1$ and $b,d,f\in Y_2$. Now each of $1,2,3,4$ has at least one neighbour in $\{a,c,e\}$, and at least one neighbour in $\{b,d,f\}$, so that none of $1,2,3,4$ can be placed in $Y$, or in $Z_1$, or in $Z_2$. It follows that none of $1,2,3,4$ can be placed in $X$ (since all neighbours of vertices in $X$ are in $Y$), and so all of $1,2,3,4$ must be in $Z_3\cup Z_4\cup Z_5$. But $1,2,3,4$ form a copy of $K_4$ in $F$, and  $Z_3\cup Z_4\cup Z_5$ induces a tripartite subgraph of $G$, a contradiction.

In this example we cannot have $F$-copies at any vertex of $X$, so the best we can do is find $\tfrac{n}{v(F)}-\Omega(p^{-1})$ vertex-disjoint copies of $F$. This may be asymptotically optimal; we have not investigated this problem. We note also that it is straightforward to generalise this construction to higher chromatic numbers $k$: we add to $F$ further numbered vertices $5,\dots,k$, adjacent to all other vertices but $r$; and we partition $Z$ into $k+1$ parts.

\bibliographystyle{abbrv}
\bibliography{references}

\end{document}